\documentclass{article}

\usepackage[latin1]{inputenc}
\usepackage{amsmath,amssymb,amsfonts,latexsym,amsthm,mathrsfs,algorithm,enumerate,xargs}
\usepackage{graphicx,color,framed}
\usepackage{algpseudocode}
\usepackage{graphicx,color,framed,gensymb,upgreek,twoopt,hyperref}
\usepackage[textsize=footnotesize]{todonotes}

\newcommand{\rmL}{\mathrm{L}}
\def\sigmaX{\mathcal{X}}
\def\sigmaY{\mathcal{Y}}

\usepackage{algorithm}
\usepackage{algpseudocode}

\newcommand{\mcg}[2][]{
\ifthenelse{\equal{#1}{}}{{\mathcal{G}_{#2}}}{\mathcal{G}_{#2}^{(#1)}}
}
\newcommand{\mcf}[2][]{
\ifthenelse{\equal{#1}{}}{{\mathcal{F}_{#2}}}{\mathcal{F}_{#2}^{(#1)}}
}
\newcommand{\mctf}[2][]{
\ifthenelse{\equal{#1}{}}{\widetilde{\mathcal{F}}_{#2}}{\widetilde{\mathcal{F}}_{#2}^{(#1)}}
}
\newcommand{\mcbf}[2][]{
\ifthenelse{\equal{#1}{}}{\overline{\mathcal{F}}_{#2}}{\overline{\mathcal{F}}_{#2}^{(#1)}}
}

\newcommand{\adjfuncforward}[1]{\vartheta_{#1}}
\newcommand{\kissforward}[3][]
{\ifthenelse{\equal{#1}{}}{p_{#2}}
{\ifthenelse{\equal{#1}{fully}}{p^{\star}_{#2}}
{\ifthenelse{\equal{#1}{smooth}}{\tilde{r}_{#2}}{\mathrm{erreur}}}}}

\newcommand{\fup}[2]{\overline{#1}_{#2}}
\newcommand{\fdown}[2]{\underline{#1}_{#2}}

\newcommand{\bdm}{\mathsf{TwoFilt}_{bdm}}
\newcommand{\fwt}{\mathsf{TwoFilt}_{fwt}}

\newcommand{\iid}{i.i.d.}

\newcommand{\eqsp}{\;}

\newcommand{\rset}{\ensuremath{\mathbb{R}}}

\newcommand{\1}{\ensuremath{\mathbf{1}}}

\newcommand{\eqdef}{\ensuremath{:=}}

\newcommand{\rmd}{\ensuremath{\mathrm{d}}}
\newcommand{\rmi}{\ensuremath{\mathrm{i}}}

\newcommand{\rme}{\ensuremath{\mathrm{e}}}

\newcommand{\esssup}[2][]
{\ifthenelse{\equal{#1}{}}{\left| #2 \right|_\infty}{\left| #2 \right|^2_{\infty}}}
\newcommand{\oscnorm}[2][]
{\ifthenelse{\equal{#1}{}}{\ensuremath{\operatorname{osc}\left(#2\right)}}{\ensuremath{\operatorname{osc}^{#1}\!\left(#2\right)}}}
\newcommand{\essosc}[3][]
{\ifthenelse{\equal{#1}{}}{\ensuremath{\operatorname{osc}_{#2}{\left(#3\right)}}}{\ensuremath{\operatorname{osc}^{#1}_{#2}\left(#3\right)}}}

\def\aux{{\scriptstyle{\mathrm{aux}}}}

\usepackage{ifthen}
\newboolean{sv_used}
\setboolean{sv_used}{true}
\usepackage{amsthm}

\newtheorem{assumptionA}{\textbf{A}\hspace{-3pt}}
\newtheorem{assumptionmix}{\textbf{H}\hspace{-3pt}}
  \newtheorem{thm}{Theorem}
  \newtheorem{prop}[thm]{Proposition}
  
  \newtheorem{lem}[thm]{Lemma}
  
  \newtheoremstyle{assumption}
  {6pt}
  {6pt}
  {}
  {-.7em}
  {\bfseries}
  {.}
  {0.1em}
  {}
  \theoremstyle{assumption}

  \theoremstyle{definition}\newtheorem{rem}{Remark}
  \theoremstyle{definition}
\theoremstyle{definition} 
  \newtheoremstyle{example}
  {6pt}% space above
  {6pt}% space below
  {\itshape}%body font
  {-.5em}%indent amount
  {\itshape}%Theorem head
  {}%punctuation
  {1.5em}%space after theorem head
  {}
  \theoremstyle{definition}

\newcommand{\PP}{\ensuremath{\mathbb{P}}}
\newcommand{\PE}{\ensuremath{\mathbb{E}}}
\newcommand{\CPE}[3][]
{\ifthenelse{\equal{#1}{}}{\mathbb{E}\left[\left. #2 \, \right| #3 \right]}{\mathbb{E}_{#1}\left[\left. #2 \, \right| #3 \right]}}
\newcommand{\CPP}[3][]
{\ifthenelse{\equal{#1}{}}{\mathbb{P}\left(\left. #2 \, \right| #3 \right)}{\mathbb{P}_{#1}\left(\left. #2 \, \right| #3 \right)}}

\renewcommand{\mid}{\,|\,}

\newcommand{\ci}[4][]%
{%
\ifthenelse{\equal{#1}{}}{\ensuremath{#2 \perp\!\!\!\perp #3 \mid #4 }}{\ensuremath{#2 \perp\!\!\!\perp #3 \mid #4 \; \: [#1]}}%
}

\newcommand{\dlim}{\ensuremath{\stackrel{\mathcal{D}}{\longrightarrow}}}
\newcommand{\plim}{\ensuremath{\stackrel{\mathrm{P}}{\longrightarrow}}}

\newcommand{\Xset}{\ensuremath{\mathsf{X}}}

\newcommand{\Xsigma}[1][]%
{%
\ifthenelse{\equal{#1}{}}{\ensuremath{\mathcal{B}(\Xset)}}{\ensuremath{\mathcal{B}(\Xset^{#1})}}
}

\newcommand{\Yset}{\ensuremath{\mathsf{Y}}}

\newcommand{\chunk}[4][]%
{\ifthenelse{\equal{#1}{}}{\ensuremath{{#2}_{#3:#4}}}{\ensuremath{#2^#1}_{#3:#4}}
}
\newcommand{\Xinit}{\ensuremath{\chi}}
\newcommand{\XinitIS}[2][]
{\ifthenelse{\equal{#1}{}}{\ensuremath{\rho_{#2}}}{\ensuremath{\check{\rho}_{#2}}}}

\newcommand{\filt}[2][]%
{%
\ifthenelse{\equal{#1}{}}{\ensuremath{\phi_{#2}}}{\ensuremath{\phi_{#1,#2}}}%
}
\newcommand{\pred}[3][]%
{%
\ifthenelse{\equal{#1}{}}{\ensuremath{\phi_{#2|#3}}}{\ensuremath{\phi_{#1,#2|#3}}}%
}

\newcommand{\adjfunc}[4][]
{\ifthenelse{\equal{#1}{}}{\ifthenelse{\equal{#4}{}}{\vartheta_{#2|#3}}{\vartheta_{#2|#3}(#4)}}
{\ifthenelse{\equal{#1}{smooth}}{\ifthenelse{\equal{#4}{}}{\tilde{\vartheta}_{#2|#3}}{\tilde{\vartheta}_{#2|#3}(#4)}}
{\ifthenelse{\equal{#1}{fully}}{\ifthenelse{\equal{#4}{}}{\vartheta^\star_{#2|#3}}{\vartheta^\star_{#2|#3}(#4)}}{\mathrm{erreur}}}}}

\newcommand{\smwght}[3]{\tilde{\omega}_{#1|#2}^{#3}}
\newcommand{\smwghtfunc}[2]{\tilde{\omega}_{#1|#2}}

\newcommandtwoopt{\post}[4][][]%
{
\ifthenelse{\equal{#2}{}}
{\ifthenelse{\equal{#1}{}}{\ensuremath{\phi_{#3|#4}}}{\ensuremath{\phi_{#1,#3|#4}}}}%
{\ifthenelse{\equal{#2}{hat}}{\ifthenelse{\equal{#1}{}}{\ensuremath{\hat{\phi}_{#3|#4}}}{\ensuremath{\hat{\phi}_{#1,#3|#4}}}}
{\ifthenelse{\equal{#2}{tar}}{\ifthenelse{\equal{#1}{}}{\ensuremath{\hat{\phi}^{\mathrm{tar}}_{#3|#4}}}{\ensuremath{\hat{\phi}^{\mathrm{tar}}_{#1,#3|#4}}}}
{\ifthenelse{\equal{#1}{}}{\ensuremath{\hat{\phi}^{#2}_{#3|#4}}}{\ensuremath{\hat{\phi}^{#2}_{#1,#3|#4}}}}
}
}
}
\newcommand{\asymVar}[3][]{
\ifthenelse{\equal{#1}{}}{\ifthenelse{\equal{#3}{}}{\ensuremath{\Gamma_{#2}}}{\ensuremath{\Gamma_{#2}\left[#3\right]}}}
{\ifthenelse{\equal{#3}{}}{\ensuremath{\Gamma_{#1,#2}}}{\ensuremath{\Gamma_{#1,#2}\left[#3\right]}}}
}

\newcommand{\asymVarest}[3][]{
\ifthenelse{\equal{#1}{}}{\ifthenelse{\equal{#3}{}}{\ensuremath{\Gamma^N_{#2}}}{\ensuremath{\Gamma^N_{#2}\left[#3\right]}}}
{\ifthenelse{\equal{#3}{}}{\ensuremath{\Gamma^N_{#1,#2}}}{\ensuremath{\Gamma^N_{#1,#2}\left[#3\right]}}}
}

\newcommand{\backasymVar}[4][]{
\ifthenelse{\equal{#1}{}}{\ifthenelse{\equal{#4}{}}{\ensuremath{\check{\Gamma}_{#2|#3}}}{\ensuremath{\check{\Gamma}_{#2|#3}\left[#4\right]}}}
{\ifthenelse{\equal{#4}{}}{\ensuremath{\check{\Gamma}_{#1,#2|#3}}}{\ensuremath{\check{\Gamma}_{#1,#2|#3}\left[#4\right]}}}
}

\newcommand{\backasymVarest}[4][]{
\ifthenelse{\equal{#1}{}}{\ifthenelse{\equal{#4}{}}{\ensuremath{\check{\Gamma}^N_{#3|#2}}}{\ensuremath{\check{\Gamma}^N_{#3|#2}\left[#4\right]}}}
{\ifthenelse{\equal{#4}{}}{\ensuremath{\check{\Gamma}^N_{#1,#3|#2}}}{\ensuremath{\check{\Gamma}^N_{#1,#3|#2}\left[#4\right]}}}
}

\newcommand{\asymVarFearnhead}[4][]{
\ifthenelse{\equal{#1}{}}{\ensuremath{\Upsilon_{#2|#3}\left[#4\right]}}{\ensuremath{\Upsilon_{#1,#2|#3}\left[#4\right]}}
}

\newcommand{\asymVarDoucet}[5][]
{\ifthenelse{\equal{#1}{}}{\Delta^{#2}_{#3|#4}\left[#5\right]}{\Delta^{#2}_{#1,#3|#4}\left[#5\right]}}

\newcommand{\asymVarDoucetest}[5][]
{\ifthenelse{\equal{#1}{}}{\Delta^{#2,N}_{#3|#4}\left[#5\right]}{\Delta^{#2,N}_{#1,#3|#4}\left[#5\right]}}

\newcommand{\asymVarJoint}[4][]{
\ifthenelse{\equal{#1}{}}{\ensuremath{\tilde \Gamma_{#2|#3}\left[#4\right]}}{\ensuremath{\tilde \Gamma_{#1,#2|#3}\left[#4\right]}}
}

\newcommand{\asymVarJointlin}[4][]{
\ifthenelse{\equal{#1}{}}{\ensuremath{\tilde \Gamma^{\mathrm{lin}}_{#2|#3}\left[#4\right]}}{\ensuremath{\tilde \Gamma^{\mathrm{lin}}_{#1,#2|#3}\left[#4\right]}}
}

\newcommandtwoopt{\incrasymVarJoint}[5][][]
{
\ifthenelse{\equal{#1}{}}
    {\ifthenelse{\equal{#2}{forward}}{\ensuremath{\tilde{\sigma}^{2,\mathrm{f}}_{#3|#4}\left[#5\right]}}{\ensuremath{\tilde{\sigma}^{2,\mathrm{b}}_{#3|#4}\left[#5\right]}}}
    {\ifthenelse{\equal{#2}{forward}}{\ensuremath{\tilde{\sigma}^{2,\mathrm{f}}_{#1,#3|#4}\left[#5\right]}}{\ensuremath{\tilde{\sigma}^{2,\mathrm{b}}_{#1,#3|#4}\left[#5\right]}}}
}

\newcommand{\incrasymVar}[4][]{
\ifthenelse{\equal{#1}{}}{\ensuremath{\sigma^2_{#2|#3}\left[#4\right]}}{\ensuremath{\sigma^2_{#1,#2|#3}\left[#4\right]}}
}

\newcommand{\logl}[2][]%
{%
\ifthenelse{\equal{#1}{}}{\ensuremath{\ell_{#2}}}{\ensuremath{\ell_{#1,#2}}}%
}
\newcommand{\lhood}[2][]%
{%
\ifthenelse{\equal{#1}{}}{\ensuremath{\mathrm{L}_{#2}}}{\ensuremath{\mathrm{L}_{#1,#2}}}%
}
\newcommand{\cc}[2][]%
{%
\ifthenelse{\equal{#1}{}}{\ensuremath{c_{#2}}}{\ensuremath{c_{#1,#2}}}%
}
\newcommand{\forvar}[2][]%
{%
\ifthenelse{\equal{#1}{}}{\ensuremath{\alpha_{#2}}}{\ensuremath{\alpha_{#1,#2}}}%
}
\newcommand{\nforvar}[2][]%
{%
\ifthenelse{\equal{#1}{}}{\ensuremath{\bar{\alpha}_{#2}}}{\ensuremath{\bar{\alpha}_{#1,#2}}}%
}

\newcommandtwoopt{\BK}[3][][]%
{%
\ifthenelse{\equal{#2}{}}{\ifthenelse{\equal{#1}{}}{\ensuremath{\mathrm{B}_{#3}}}{\ensuremath{\mathrm{B}_{#1,#3}}}}%
{\ifthenelse{\equal{#1}{}}{\ensuremath{\hat{\mathrm{B}}_{#3}}}{\ensuremath{\hat{\mathrm{B}}_{#1,#3}}}}
}
\newcommandtwoopt{\bk}[3][][]%
{%
\ifthenelse{\equal{#2}{}}{\ifthenelse{\equal{#1}{}}{\ensuremath{\mathrm{b}_{#3}}}{\ensuremath{\mathrm{b}_{#1,#3}}}}%
{\ifthenelse{\equal{#1}{}}{\ensuremath{\hat{\mathrm{b}}_{#3}}}{\ensuremath{\hat{\mathrm{b}}_{#1,#3}}}}
}

\newcommand{\filtfunc}[2][]%
{%
\ifthenelse{\equal{#1}{}}{\ensuremath{\tau_{#2}}}{\ensuremath{\tau_{#1,#2}}}%
}
\newcommand{\NISE}[4][]%
{%
\ifthenelse{\equal{#1}{}}{\ensuremath{\tilde{#2}^{\scriptstyle \mathrm{IS}}_{#4}\left(#3 \right)}}{\ensuremath{\tilde{#2}^{\scriptstyle \mathrm{IS}}_{#1,#4}\left(#3 \right)}}%
}
\newcommand{\ISE}[4][]%
{%
\ifthenelse{\equal{#1}{}}{\ensuremath{\widehat{#2}^{\scriptstyle  \mathrm{IS}}_{#4} \left(#3 \right)}}{\ensuremath{\widehat{#2}^{\scriptstyle  \mathrm{IS}}_{#1,#4} \left(#3 \right)}}%
}
\newcommand{\SIRE}[4][]%
{%
\ifthenelse{\equal{#1}{}}{\ensuremath{\hat{#2}^{\scriptstyle  \mathrm{SIR}}_{#4} \left(#3 \right)}}{\ensuremath{\hat{#2}^{\scriptstyle  \mathrm{SIR}}_{#1,#4} \left(#3 \right)}}%
}
\newcommand{\MCE}[3]%
{
{\ensuremath{\hat{#1}^{\scriptstyle  \mathrm{MC}}_{#3} \left(#2 \right)}}%
}
\newcommand{\kiss}[3][]
{\ifthenelse{\equal{#1}{}}{r_{#2|#3}}
{\ifthenelse{\equal{#1}{fully}}{r^{\star}_{#2|#3}}
{\ifthenelse{\equal{#1}{smooth}}{\tilde{r}_{#2|#3}}{\mathrm{erreur}}}}}

\newcommand{\fakeprior}{\gamma}
\def\forward{\mathrm{f}}
\def\backward{\mathrm{b}}

\newcommand{\epart}[2]{\ensuremath{\xi_{#1}^{#2}}}
\newcommand{\smpart}[3]{\ensuremath{\tilde{\xi}_{#1|#2}^{#3}}}
\newcommand{\epartpred}[2]{\ensuremath{\xi_{#1}^{#2}}}

\newcommand{\ewght}[2]{\ensuremath{\omega_{#1}^{#2}}}
\newcommand{\ewghtfunc}[1]{\ensuremath{\omega_{#1}}}

\newcommand{\swght}[2]{\ensuremath{\omega_{#1}^{#2}}}

\newcommand{\sumwght}[2][]{%
\ifthenelse{\equal{#1}{}}{\ensuremath{\Omega_{#2}}}{\ensuremath{\Omega_{#2}^{(#1)}}}}
\newcommand{\tsumweight}[2][]{%
\ifthenelse{\equal{#1}{}}{\ensuremath{\widetilde{\Omega}_{#2}}}{\ensuremath{\widetilde{\Omega}_{#2}^{(#1)}}}}
\newcommand{\indexresidual}[2][]{%
\ifthenelse{\equal{#1}{}}{\ensuremath{J_{#2}}}{\ensuremath{J_{#2}^{(#1)}}}}

\newcommand{\efilt}[2][]%
{%
\ifthenelse{\equal{#1}{}}{\ensuremath{\hat{\phi}_{#2}}}{\ensuremath{\hat{\phi}_{#1,#2}}}%
}
\newcommand{\epost}[3][]%
{%
\ifthenelse{\equal{#1}{}}{\ensuremath{\hat{\phi}_{#2|#3}}}{\ensuremath{\hat{\phi}_{#1,#2|#3}}}%
}

\newcommandtwoopt{\backDist}[4][][]%
{
\ifthenelse{\equal{#1}{}}{\ifthenelse{\equal{#2}{}}{\psi_{#3|#4}}{\psi_{#2,#3|#4}}}%
{\ifthenelse{\equal{#1}{hat}}{\ifthenelse{\equal{#2}{}}{\hat{\psi}_{#3|#4}}{\hat{\psi}_{#2,#3|#4}}}
{\ifthenelse{\equal{#1}{tar}}{\ifthenelse{\equal{#2}{}}{\hat{\psi}^{\mathrm{tar}}_{#3|#4}}{\hat{\psi}^{\mathrm{tar}}_{#2,#3|#4}}}
{\ifthenelse{\equal{#1}{aux}}{\ifthenelse{\equal{#2}{}}{\hat{\psi}^{\mathrm{aux}}_{#3|#4}}{\hat{\psi}^{\mathrm{aux}}_{#2,#3|#4}}}
}
}
}
}

\newcommand{\backweight}{\ensuremath{{\check{\omega}}}}

\newcommand{\ebackinit}{\ensuremath{{\check{\rho}}}}

\newcommand{\ebackpart}[3]{\ensuremath{{\check{\xi}}_{#1|#2}^{#3}}}
\newcommand{\ebackwght}[3]{\ensuremath{{\check{\omega}}_{#1|#2}^{#3}}}
\newcommand{\ebackwghtfunc}[2]{\ensuremath{{\check{\omega}}_{#1|#2}}}
\newcommand{\backsumwght}[3][]{%
\ifthenelse{\equal{#1}{}}{\ensuremath{{\check{\Omega}}_{#2|#3}}}{\ensuremath{{\check{\Omega}}_{#2|#3}^{(#1)}}}}

\newcommand{\instrpostaux}[2]{\ensuremath{\pi_{#1|#2}}}

\newcommandx\functionset[3][1=,3=]{
\ifthenelse{\equal{#1}{c}}
{\mathrm{C}(\mathsf{#2})}%fonctions continues
{\ifthenelse{\equal{#1}{bc}}{\mathrm{C}^{#3}_b(\mathsf{#2})}%fonctions continues born\'{e}es
{\ifthenelse{\equal{#1}{u}}{\mathrm{U}^{#3}(\mathsf{#2})}%fonctions uniform\'{e}ment continues
{\ifthenelse{\equal{#1}{bu}}{\mathrm{U}_b^{#3}(\mathsf{#2})}%fonctions uniform\'{e}ment continues born\'{e}es
{\ifthenelse{\equal{#1}{l}}{\mathrm{Lip}_{#3}(\mathsf{#2})}%fonctions lipschitz
{\ifthenelse{\equal{#1}{cz}}{\mathrm{C}_0^{#3}(\mathsf{#2})}
{\ifthenelse{\equal{#1}{k}}{\mathrm{C}_c^{#3}(\mathsf{#2})} %fonctions continues à support compact
{\ifthenelse{\equal{#1}{bl}}{\mathrm{Lip}_{b#3}(\mathsf{#2})}%fonctions lipschitz born\'{e}es
{\ifthenelse{\equal{#1}{ldp}}{\mathrm{Lip}_{#3}(\mathsf{#2})}%fonctions lipschitz à poids
{\mathbb{F}_{#1}^{#3}(\mathsf{#2},\mathcal{#2})}%le reste
}}}}}}}}}

\usepackage{geometry}
\usepackage{fancyhdr}
\begin{document}

\author{Thi Ngoc Minh Nguyen\footnote{LTCI, CNRS and T\'el\'ecom ParisTech.}\and Sylvain Le {C}orff\footnote{Laboratoire de Math\'ematiques d'Orsay, Univ. Paris-Sud, CNRS, Universit\'e Paris-Saclay.}\and Eric Moulines\footnote{Centre de Math\'ematiques Appliqu\'ees, Ecole Polytechnique.}}

\title{On the two-filter approximations of marginal smoothing distributions in general state space models}

\date{}
\lhead{Nguyen et al.}
\rhead{Two-filter approximations of marginal smoothing distributions}

\maketitle

\begin{abstract}
A prevalent problem in general state space models is the approximation of the smoothing distribution of a state
conditional on the observations from the past, the present, and the future. The aim of this paper is to provide a rigorous analysis of such approximations of smoothed distributions provided by the two-filter algorithms. We extend the results available for the approximation of smoothing distributions to these two-filter approaches which combine a forward filter approximating the filtering distributions with a backward information filter approximating a quantity proportional to the posterior distribution of the state given future observations.
\end{abstract}

\section{Introduction}
State-space models play a key role in a large variety of disciplines such as engineering, econometrics, computational biology or signal processing, see \cite{doucet:defreitas:gordon:2001,douc:moulines:stoffer:2014} and references therein. This paper provides a nonasymptotic analysis of a Sequential Monte Carlo Method (SMC) which aims at performing optimal smoothing in nonlinear and non Gaussian state space models. Given two measurable spaces $(\Xset,\sigmaX)$ and $(\Yset,\sigmaY)$, consider a bivariate stochastic process $\{(X_{t},Y_{t})\}_{t\geq 0}$ taking values in the product space $(\Xset\times\Yset,\sigmaX\otimes\sigmaY)$, where the hidden state sequence $\{X_{t}\}_{t\geq 0}$ is observed only through the observation process $\{Y_{t}\}_{t\geq 0}$.
Statistical inference in general state space models usually involves the computation  of conditional distributions of some unobserved states given a set of observations. These posterior distributions are crucial to compute smoothed expectations of additive functionals which appear naturally for maximum likelihood parameter inference in hidden Markov models (computation of the Fisher score or of the intermediate quantity of the Expectation Maximization algorithm), see \cite[Chapter $10$ and $11$]{cappe:moulines:ryden:2005}, \cite{kantas:doucet:signh:2015,doucet:poyiadjis:singh:2011,lecorff:fort:2013a,lecorff:fort:2013b}.

Nevertheless, exact computation of the filtering and smoothing distributions is possible only for linear and Gaussian state spaces or when the state space $\Xset$ is finite. This paper focuses on particular  instances of Sequential Monte Carlo methods which approximate sequences of distributions in a general state space $\Xset$ with random samples, named particles, associated with nonnegative importance weights. Those particle filters and smoothers rely on the combination of sequential importance sampling steps to propagate particles in the state space and importance resampling steps to duplicate or discard particles according to their importance weights. The first implementation of these SMC methods, introduced in \cite{gordon:salmond:smith:1993,kitagawa:1996}, propagates the particles using the Markov kernel of the hidden process $\{X_{t}\}_{t\geq 0}$ and uses a multinomial resampling step based on the importance weights to select particles at each time step. An interesting feature of this {\em Poor man's smoother} is that it provides an approximation of the joint smoothing distribution by storing the ancestral line of each particle with a complexity growing only linearly with the number $N$ of particles, see for instance \cite{delmoral:2004}. However, this smoothing algorithm has a major shortcoming since the successive resampling steps induce  an important  depletion of the particle trajectories. This degeneracy of the particle sequences leads to trajectories sharing a common ancestor path; see \cite{doucet:poyiadjis:singh:2011,jacob:murray:rubenthaler:2013} for a discussion.

Approximations of the smoothing distributions may also be obtained using the forward filtering backward smoothing decomposition in general state space models.  The Forward Filtering Backward Smoothing algorithm (FFBS) and  the Forward Filtering Backward Simulation algorithm (FFBSi) developed respectively in \cite{kitagawa:1996,huerzeler:kunsch:1998,doucet:godsill:andrieu:2000} and \cite{godsill:doucet:west:2004} avoid the path degeneracy issue of the {\em Poor man's smoother} at the cost of a computational complexity growing with $N^2$. Both algorithms rely on a forward pass which produces a set of particles and weights approximating the sequence of filtering distributions up to time $T$. Then, the backward pass of the FFBS algorithm modifies all the weights computed in the forward pass according to the so-called  backward decomposition of the smoothing distribution keeping all the particles fixed. On the other hand, the FFBSi  algorithm samples independently particle trajectories among all the possible paths produced by the forward pass. It is  shown in \cite{mongillo:deneve:2008,cappe:2011,delmoral:doucet:singh:2010} that the FFBS algorithm can be implemented using only a forward pass when approximating smoothed expectations of additive functionals but with a complexity still growing quadratically with $N$.  Under the mild assumption that the transition density of the hidden chain $\{X_{t}\}_{t\geq 0}$ is uniformly bounded above, \cite{douc:garivier:moulines:olsson:2011} proposed an accept-reject mechanism to implement the FFBSi algorithm with a complexity growing only linearly with $N$. Concentration inequalities,  controls of the $\rmL_{q}$-norm of the deviation between smoothed functionals and their approximations and Central Limit Theorems (CLT) for the FFBS and the FFBSi algorithms have been established in \cite{delmoral:doucet:singh:2010,douc:garivier:moulines:olsson:2011,dubarry:lecorff:2013}.

Recently, \cite{olsson:westerborn:2015} proposed a new SMC algorithm, the particle-based rapid incremental smoother (PaRIS), to approximate online, using only a forward pass, smoothed expectations of additive functionals. The crucial feature of this algorithm is that its complexity grows only linearly with $N$ as it samples on-the-fly particles distributed according to the backward dynamics of the hidden chain conditionally on the observations $Y_0,\ldots,Y_T$.  The authors show concentration inequalities and CLT for the estimators provided by the PaRIS algorithm.

In this paper, we extend the theoretical results available for the SMC approximations of smoothing distributions to the estimators given by the two-filter algorithms. These methods were first introduced in the particle filter literature by \cite{kitagawa:1996} and developed further by \cite{briers:doucet:maskell:2010} and \cite{fearnhead:wyncoll:tawn:2010}. The two-filter approach combines the output of two independent filters, one that evolves forward in time and approximates the filtering distributions and another that evolves backward in time approximating a quantity proportional to the posterior distribution of a state given future observations. In \cite{fearnhead:wyncoll:tawn:2010}, the authors introduced a proposal mechanism leading to algorithms whose complexity grows linearly with the number of particles.  An algorithm similar to the algorithm of  \cite{briers:doucet:maskell:2010} may also be implemented with an $\mathcal{O}(N)$ computational complexity following the same idea. We analyze all these algorithms which approximate the marginal smoothing distributions (smoothing distributions of one state given all the observations) and provide concentration inequalities as well as CLT.

This paper is organized as follows. Section~\ref{sec:TwoFilters} introduces the different particle approximations of the marginal smoothing distributions given by the two-filter algorithms. Sections~\ref{sec:ExponentialTwoFilters} and~\ref{sec:CLTTwoFilters} provide exponential deviation inequalities and CLT for the particle approximations under mild assumptions on the hidden Markov chain. Under additional {\em strong mixing assumptions}, it is shown that the results of Section~\ref{sec:ExponentialTwoFilters} are uniform in time and that the asymptotic variance in Section~\ref{sec:CLTTwoFilters} may be uniformly bounded in time. All proofs are postponed to Section~\ref{sec:proofs}.

\subsection*{Notations and conventions}
Let $\Xset$ and $\Yset$ be two general state-spaces endowed with countably generated $\sigma$-fields $\sigmaX$ and $\sigmaY$. $\functionset[b]{X}$ is the set of all real valued bounded measurable functions on $(\Xset,\sigmaX)$. $Q$ is a Markov transition kernel defined on $\Xset\times\sigmaX$ and $\{g_{t}\}_{t\geq 0}$ a family of positive functions defined on $\Xset$. For any $x \in \Xset$, $Q(x,\cdot)$ has a density $q(x, \cdot)$ with respect to a measure $\lambda$ on $(\Xset,\sigmaX)$.  The oscillation of a real valued function defined on a space $\mathsf{Z}$ is given by:
$\oscnorm{h} \eqdef \sup_{z,z'\in\mathsf{Z}}\left|h(z)-h(z')\right|$.

\section{The two-filter algorithms}
\label{sec:TwoFilters}
For any measurable function $h$ on $\Xset^{t-s+1}$, probability distribution $\chi$ on $(\Xset,\sigmaX)$, $T\geq 0$ and $0 \leq s \leq t \leq T$, define the joint smoothing distribution by:
\begin{equation}
\label{eq:smooth}
\post[\Xinit]{s:t}{T}[h] \eqdef \frac{\int \chi(\rmd x_0)g_{0}(x_0)\prod_{u=1}^{T}Q(x_{u-1},\rmd x_u)g_{u}(x_u)h(x_{s:t})}{\int \chi(\rmd x_0)g_{0}(x_0)\prod_{u=1}^{T}Q(x_{u-1},\rmd x_u)g_{u}(x_u)}\eqsp,
\end{equation}
where  $a_{u:v}$ is a short-hand notation for $\{a_s\}_{s=u}^{v}$. In the following we use the notations $\post[\Xinit]{s}{T}\eqdef \post[\Xinit]{s:s}{T}$ and $\filt{\Xinit,t}\eqdef \post[\Xinit]{t:t}{t}$. The aim of this paper is to provide a rigorous analysis of the performance of SMC algorithms approximating the sequence $\post[\Xinit]{s}{T}$ for $0\le s\le T$. The algorithms analyzed in this paper are based on the {\em two-filter formula} introduced in \cite{briers:doucet:maskell:2010,fearnhead:wyncoll:tawn:2010}, which we now detail.
\subsection{Forward filter}
Let $\{\epart{0}{\ell}\}_{\ell = 1}^N$ be \iid\ and  distributed according to the instrumental distribution $\rho_0$ and define the importance weights
\[
\ewght{0}{\ell} \eqdef \frac{\rmd \Xinit}{\rmd \XinitIS{0}}(\epart{0}{\ell}) \, g_{0}(\epart{0}{\ell})\eqsp.
\]
For any  $h\in \functionset[b]{X}$,
\[
\filt{\Xinit,0}^N[h]\eqdef \sumwght{0}^{-1}\sum_{\ell=1}^N \ewght{0}{\ell}h(\epart{0}{\ell})\eqsp,\quad\mbox{where}\quad\sumwght{0}\eqdef \sum_{\ell=1}^N \ewght{0}{\ell}\eqsp,
\]
is a consistent estimator of $\filt{\Xinit,0}[h]$, see for instance \cite{delmoral:2004}. Then, based on $\{(\epart{s-1}{\ell},\ewght{s-1}{\ell})\}_{\ell=1}^N$ a new set of particles and importance weights is obtained using the auxiliary sampler introduced in \cite{pitt:shephard:1999}. Pairs $\{ (I_s^{\ell}, \epart{s}{\ell}) \}_{\ell = 1}^N$ of indices and particles are simulated independently from the instrumental distribution with density on $\{1, \dots, N\} \times \Xset$:
\begin{equation} 
\label{eq:instrumental-distribution-filtering}
\instrpostaux{s}{s}(\ell,x) \propto \ewght{s-1}{\ell} \adjfuncforward{s}(\epart{s-1}{\ell}) \kissforward{s}{s}(\epart{s-1}{\ell},x) \eqsp,
\end{equation}
where $\adjfuncforward{s}$ is the adjustment multiplier weight function and $\kissforward{s}{s}$ is a Markovian transition density. For any  $\ell \in\{1, \dots, N\}$, $\epart{s}{\ell}$ is associated with the  importance weight defined by:
\begin{equation}
\label{eq:weight-update-filtering}
    \ewght{s}{\ell} \eqdef \frac{q(\epart{s-1}{I_s^{\ell}},\epart{s}{\ell}) g_{s}(\epart{s}{\ell})}{\adjfuncforward{s}(\epart{s-1}{I_s^{\ell}}) \kissforward{s}{s}(\epart{s-1}{I_s^{\ell}},\epart{s}{\ell})}
\end{equation}
to produce the following approximation of $\filt{\Xinit,s}[h]$:
\[
\filt{\Xinit,s}^N[h]\eqdef \sumwght{s}^{-1}\sum_{\ell=1}^N \ewght{s}{\ell}h(\epart{s}{\ell})\eqsp,\quad \mbox{where}\quad \sumwght{s}\eqdef \sum_{\ell=1}^N \ewght{s}{\ell}\eqsp.
\]
\subsection{Backward filter}
Let $\{\gamma_t\}_{t \geq 0}$ be a family of positive measurable functions such that, for all $t\in \{0, \ldots, T\}$,
\begin{equation}\label{eq:defGamma}
\int \gamma_t(x_t) \, \rmd x_t \left[\prod_{u = t+1}^T g_{u-1}(x_{u-1}) \, Q(x_{u-1}, \rmd x_u) \right] g_T(x_T) < \infty \eqsp.
\end{equation}
Following \cite{briers:doucet:maskell:2010}, for any $0\le t\le T$ we introduce the backward filtering distribution $\backDist[][\gamma]{t}{T}$ on $\sigmaX$ (referred to as the  \emph{backward information filter} in \cite{kitagawa:1996} and \cite{briers:doucet:maskell:2010}) defined, for any $h\in \functionset[b]{X}$, by:
\[
\backDist[][\gamma]{t}{T}[h] \eqdef \\
\frac{\int \gamma_t(x_t) \, \rmd x_t\left[\prod_{u=t+1}^T g_{u-1}(x_{u-1}) \, Q(x_{u-1},\rmd x_u)\right] g_T(x_T) h(x_t)}{\int \gamma_t(x_t) \, \rmd x_t \left[ \prod_{u=t+1}^T g_{u-1}(x_{u-1}) \, Q(x_{u-1},\rmd x_u )\right] g_T(x_T)} \eqsp.
\]
If the distribution of $X_t$ has probability density function $\gamma_t$, then $\backDist[][\gamma]{t}{T}$ is the conditional distribution of $X_t$ given $\chunk{Y}{t}{T}$.  Contrary to \cite{briers:doucet:maskell:2010} or \cite{fearnhead:wyncoll:tawn:2010}, $\int \gamma_t(x_t)\rmd x_t$ may be infinite. The only requirement about the nonnegative functions $\{\gamma_t\}_{t\geq 0}$ is the condition \eqref{eq:defGamma} and the fact that $\gamma_t$ should be available in closed form. Here $\gamma_t$ is a possibly improper prior introduced to make $\backDist[][\gamma]{t}{T}$ a proper posterior distribution, which is of key importance when producing particle approximations of such quantities.
For $0\le t \le T-1$, the backward information filter is computed by the recursion
\begin{equation}
\backDist[][\gamma]{t}{T}[h]\propto \int \backDist[][\gamma]{t+1}{T}(\rmd x_{t+1})\left[\gamma_t(x_t) g_t(x_t) \frac{q(x_t,x_{t+1})}{\gamma_{t+1}(x_{t+1})} \right] h(x_t) \, \rmd x_t \eqsp, \label{eq:recurBack}
\end{equation}
in the backward time direction. \eqref{eq:recurBack} is analogous to the forward filter recursion and particle approximations of the backward information filter can be obtained similarly.
Using the definition of the forward filtering distribution at time $s - 1$ and the backward information filter at time $s + 1$, the marginal smoothing distribution may be expressed as
\begin{equation}
\label{eq:filtBackfilt}
\post[\Xinit]{s}{T}[h]  \propto \int \filt{\Xinit,s-1}(\rmd x_{s-1}) \backDist[][\gamma]{s+1}{T}(\rmd x_{s+1})q(x_{s-1},x_s) g_s(x_s) \frac{q(x_s,x_{s+1})}{\gamma_{s+1}(x_{s+1})} h(x_s) \rmd x_s \eqsp.
\end{equation}
 We now describe the Sequential Monte Carlo methods used to approximate the recursion \eqref{eq:recurBack} in \cite{briers:doucet:maskell:2010}, \cite{fearnhead:wyncoll:tawn:2010}.
Let $\ebackinit_T$ be an instrumental probability density on $\Xset$ and $\{\ebackpart{T}{T}{i}\}_{i=1}^N$ be i.i.d. random variables such that $\ebackpart{T}{T}{i}\sim\ebackinit_T$ and define
\[
\ebackwght{T}{T}{i} \eqdef \frac{g_T(\ebackpart{T}{T}{i})\gamma_T(\ebackpart{T}{T}{i})}{\ebackinit_T(\ebackpart{T}{T}{i})}\eqsp.
\]
Let now $\{ (\ebackpart{t+1}{T}{i},\ebackwght{t+1}{T}{i}) \}_{i=1}^N$ be a weighted sample targeting the backward information filter distribution $\backDist[][\gamma]{t+1}{T}[h]$  at time $t+1$:
\[
\backDist[][\gamma]{t+1}{T}^N[h] \eqdef \backsumwght{t+1}{T}^{-1} \sum_{i=1}^N \ebackwght{t+1}{T}{i}h(\ebackpart{t+1}{T}{i})\eqsp,\quad\mbox{where}\quad\backsumwght{t+1}{T} \eqdef \sum_{i=1}^N \ebackwght{t+1}{T}{i}\eqsp.
\]
Plugging this approximation into \eqref{eq:recurBack} yields the target probability density
\begin{equation*}
\backDist[tar][\gamma]{t}{T}(x_t)\propto \sum_{i=1}^N \ebackwght{t+1}{T}{i} \left[\gamma_t(x_t) g_t(x_t) \frac{q(x_t,\ebackpart{t+1}{T}{i})}{\gamma_{t+1}(\ebackpart{t+1}{T}{i})} \right] \eqsp,
\end{equation*}
which is the marginal probability density function of  $x_t$ of the joint  density
\begin{equation*}
\backDist[aux][\gamma]{t}{T} ((i, x_t) \propto \frac{\ebackwght{t+1}{T}{i}}{\gamma_{t+1}(\ebackpart{t+1}{T}{i})} \gamma_t(x_t) g_t(x_t) q(x_t, \ebackpart{t+1}{T}{i}) \eqsp.
\end{equation*}
A particle approximation of the backward information filter at time $t$ can be derived by choosing an adjustment weight function $\adjfunc{t}{T}{}$ and an instrumental density kernel $\kiss{t}{T}$, and simulating $\{ (\check{I}^i_t, \ebackpart{t}{T}{i}) \}_{i=1}^N$ from
the instrumental probability density on $\{1,\ldots,N\}\times\Xset$ given by
\begin{equation}
\label{eq:ebackparticle-update-backward-filtering}
\instrpostaux{t}{T}(i,x_t)\propto  \frac{\ebackwght{t+1}{T}{i} \adjfunc{t}{T}{\ebackpart{t+1}{T}{i}}}{\gamma_{t+1}(\ebackpart{t+1}{T}{i})} \kiss{t}{T}(\ebackpart{t+1}{T}{i},x_t) \eqsp.
\end{equation}
Subsequently, the particles are associated with the importance weights
\begin{equation}
\label{eq:ebackweight-update-backward-filtering}
\ebackwght{t}{T}{i} \eqdef \frac{\gamma_t(\ebackpart{t}{T}{i}) g_t(\ebackpart{t}{T}{i}) q(\ebackpart{t}{T}{i},\ebackpart{t+1}{T}{\check{I}^i_t})}{\adjfunc{t}{T}{\ebackpart{t+1}{T}{\check{I}^i_t}} \kiss{t}{T}(\ebackpart{t+1}{T}{\check{I}^i_t}, \ebackpart{t}{T}{i})} \eqsp.
\end{equation}
Ideally, a fully adapted version of the auxiliary backward information filter is obtained by using the adjustment weights $\adjfunc[fully]{t}{T}{x}=  \int \gamma_t(x_t) g_t(x_t) q(x_t,x) \, \rmd x_t $ and the proposal kernel density
\[
\kiss[fully]{t}{T}(x, x_t)=\gamma_t(x_t) g_t(x_t) \frac{q(x_t,x)}{\adjfunc[fully]{t}{T}{x}} \eqsp,
\]
yielding uniform importance weights. Such a solution is most likely to be cumbersome from a computational perspective.

\subsection{Two-filter approximations of the marginal smoothing distributions}
Plugging the particle approximations of the forward and backward filter distributions into \eqref{eq:filtBackfilt} provides the following mixture approximation of the smoothing distribution:
\begin{equation}
\label{eq:definition-smoothing-target}
\post[\Xinit][tar]{s}{T}(x_s) \propto \sum_{i=1}^N \sum_{j=1}^N \frac{\ewght{s-1}{i} \ebackwght{s+1}{T}{j}}{\gamma_{s+1}(\ebackpart{s+1}{T}{j})}  q(\epart{s - 1}{i},x_s) g_s(x_s) q(x_s, \ebackpart{s+1}{T}{j}) \eqsp.
\end{equation}
Following the $\fwt$ algorithm of Fearnhead, Wyncoll and Tawn \cite{fearnhead:wyncoll:tawn:2010}, the probability density \eqref{eq:definition-smoothing-target} might be seen as
the marginal density of $x_s$ obtained from the joint density on the product space $\{1,\dots,N\}^2 \times \Xset$ given by
\begin{equation}
\label{eq:definition-smoothing-auxiliary}
\post[\chi][\aux]{s}{T}(i,j, x_s) \propto \frac{\ewght{s-1}{i} \ebackwght{s+1}{T}{j}}{\gamma_{s+1}(\ebackpart{s+1}{T}{j})} q(\epart{s-1}{i},x_s) g_s(x_s) q(x_s,\ebackpart{s+1}{T}{j})\eqsp.
\end{equation}
The $\fwt$ algorithm draws a set $\{(I_{s}^\ell,\check{I}_{s}^\ell,\smpart{s}{T}{\ell}) \}_{\ell=1}^N$ of indices and particle positions from the instrumental density
\begin{equation}
\label{eq:auxiliary-proposal-smoothing}
\instrpostaux{s}{T}(i,j,x_s) \propto  \frac{\ewght{s-1}{i} \adjfunc[smooth]{s}{T}{\epart{s-1}{i},\ebackpart{s+1}{T}{j}} \ebackwght{s+1}{T}{j}}{\gamma_{s+1}(\ebackpart{s+1}{T}{j})} \kiss[smooth]{s}{T}(\epart{s-1}{i},\ebackpart{s+1}{T}{j}; x_s)\eqsp,
\end{equation}
where, as above, $\adjfunc[smooth]{s}{T}{x,x'}$ is an adjustment multiplier weight function (which now depends on the forward and backward particles) and $\kiss[smooth]{s}{T}$ is an instrumental kernel. We then associate with each draw $(I_{s}^\ell,\check{I}_{s}^\ell, \smpart{s}{T}{\ell})$ the importance weight
\begin{equation}
\label{eq:weight-fearnhead}
\smwght{s}{T}{\ell} \eqdef \frac{q(\epart{s-1}{I_{s}^\ell},\smpart{s}{T}{\ell}) g_s(\smpart{s}{T}{\ell}) q(\smpart{s}{T}{\ell},\ebackpart{s+1}{T}{\check{I}_{s}^\ell})}{\adjfunc[smooth]{s}{T}{\epart{s-1}{I_{s}^\ell},\ebackpart{s+1}{T}{\check{I}_{s}^\ell}}
\kiss[smooth]{s}{T}(\epart{s-1}{I_{s}^\ell},\ebackpart{s+1}{T}{\check{I}_{s}^\ell}; \smpart{s}{T}{\ell})}\eqsp,\quad \tilde{\Omega}_{s|T}\eqdef \sum_{\ell=1}^N\smwght{s}{T}{\ell}  \eqsp.
\end{equation}
Then, the auxiliary indices $\{(I_{s}^\ell,\check{I}_{s}^\ell)\}_{\ell=1}^N$ are discarded and $\{ (\smwght{s}{T}{\ell},\smpart{s}{T}{\ell}) \}_{\ell=1}^N$ approximate the target smoothing density $\post[\Xinit][tar]{s}{T}$. Mimicking the arguments in \cite{huerzeler:kunsch:1998} and further developed in \cite{kuensch:2005}, the auxiliary particle filter is fully adapted if the adjustment weight function is $\adjfunc[fully]{s}{T}{x,x'} = \int q(x,x_s) g_s(x_s) q(x_s,x') \, \rmd x_s$ and the instrumental kernel is
\[
\kiss[fully]{s}{T}\left(x,x'; x_s\right) = q(x,x_s) g_s(x_s) q(x_s,x')/\adjfunc[fully]{s}{T}{x,x'}\eqsp.
\]
 Except in simple scenarios, simulating from the fully adapted auxiliary filter is computationally intractable.

Instead of considering the target distribution \eqref{eq:definition-smoothing-target} as the marginal of the auxiliary distribution \eqref{eq:definition-smoothing-auxiliary}
over pairs of indices, the $\bdm$ algorithm of  \cite{briers:doucet:maskell:2010} uses the following \emph{partial} auxiliary distributions having densities,
\begin{align*}
& \post[][\aux,\mathrm{f}]{s}{T}(i,x_s) \propto  \ewght{s-1}{i} q(\epart{s-1}{i}, x_s) g_s(x_s) \sum_{j=1}^N \frac{\ebackwght{s+1}{T}{j}}{\gamma_{s+1}(\ebackpart{s+1}{T}{j})} q(x_s, \ebackpart{s+1}{T}{j}) \eqsp,\\
& \post[][\aux, \mathrm{b}]{s}{T}(j,x_s) \propto  \frac{\ebackwght{s+1}{T}{j}}{\gamma_{s+1}(\ebackpart{s+1}{T}{j})} q(x_s, \ebackpart{s+1}{T}{j}) g_s(x_s) \sum_{i=1}^N \ewght{s-1}{i} q(\epart{s-1}{i},x_s) \eqsp.
\end{align*}
Since $\post[\Xinit][tar]{s}{T}$ is the marginal probability density of the partial auxiliary distributions $\post[][\aux,\mathrm{f}]{s}{T}$ and $\post[][\aux,\mathrm{b}]{s}{T}$ with respect to the forward and the backward particle indices, respectively, we may sample from $\post[\Xinit][tar]{s}{T}$ by simulating instead $\{(I_s^\ell,\epart{s}{\ell}) \}_{\ell=1}^N$ or $\{(\check{I}_{s}^\ell,\ebackpart{s}{T}{\ell}) \}_{\ell=1}^N$ from the instrumental probability density functions
\begin{align*}
&\instrpostaux{s}{T}^{\mathrm{f}}(i,x_s) \propto  \ewght{s-1}{i} \adjfuncforward{s}(\epart{s-1}{i}) \kissforward{s}{s}(\epart{s-1}{i}, x_s) \eqsp, \\
&\instrpostaux{s}{T}^{\mathrm{b}}(j,x_s) \propto  \adjfunc{s}{T}{\ebackpart{s+1}{T}{j}} \ebackwght{s+1}{T}{j} \kiss{s}{T}(\ebackpart{s+1}{T}{j}, x_s)/\gamma_{s+1}(\ebackpart{s+1}{T}{j}) \eqsp,
\end{align*}
where $(\adjfuncforward{s},\kissforward{s}{s})$ and $(\adjfunc{s+1}{T}{},\kiss{s}{T})$ are the adjustment multiplier weight functions and the instrumental kernels used in the forward and backward passes. In this case the algorithm uses the particles obtained when approximating the forward filter and backward information filter to provide two different weighted samples
$\{ (\smwght{s}{T}{i,\mathrm{f}}, \epart{s}{i}) \}_{i=1}^N$ and $\{ (\smwght{s}{T}{i,\mathrm{b}}, \ebackpart{s}{T}{i}) \}_{i=1}^N$ targeting the marginal smoothing distribution, where the
\emph{forward} $\{ \smwght{s}{T}{i, \mathrm{f}} \}_{i=1}^N$ and \emph{backward} $\{ \smwght{s}{T}{i, \mathrm{b}} \}_{i=1}^N$ importance weights are given by
\begin{align}
\label{eq:definition-smoothing-weight-forward}
&\smwght{s}{T}{i, \mathrm{f}} \eqdef \ewght{s}{i} \sum_{j=1}^N \ebackwght{s+1}{T}{j} q(\epart{s}{i},\ebackpart{s+1}{T}{j})/\gamma_{s+1}(\ebackpart{s+1}{T}{j})\eqsp,\quad \tilde{\Omega}^{\mathrm{f}}_{s|T}\eqdef\sum_{j=1}^N \smwght{s}{T}{j, \mathrm{f}}\eqsp, \\
\label{eq:definition-smoothing-weight-backward}
&\smwght{s}{T}{j, \mathrm{b}} \eqdef \ebackwght{s}{T}{j} \sum_{i=1}^N \ewght{s-1}{i} q(\epart{s-1}{i},\ebackpart{s}{T}{j})/\gamma_{s}(\ebackpart{s}{T}{j})\eqsp,\quad \hspace{.8cm}\tilde{\Omega}^{\mathrm{b}}_{s|T}\eqdef\sum_{j=1}^N \smwght{s}{T}{j, \mathrm{b}} \eqsp.
\end{align}
An important drawback of these algorithms is that the computation of the forward and backward importance weights grows quadratically with the number $N$ of particles.

\subsection{$\mathcal{O}(N)$ approximations of the marginal smoothing distributions}
In \cite{fearnhead:wyncoll:tawn:2010}, the authors introduced a proposal mechanism in \eqref{eq:auxiliary-proposal-smoothing} such that the indices $(I_{s},\check{I}_{s})$ of the forward and backward particles chosen at time $s-1$ and $s+1$ are sampled independently.  Such choices lead to algorithms whose complexity grows linearly with the number of particles. The $\mathcal{O}(N)$ algorithm displayed in  \cite{fearnhead:wyncoll:tawn:2010} suggests to use an adjustment multiplier weight function in \eqref{eq:auxiliary-proposal-smoothing} such that $I_{s}$ and $\check{I}_{s}$ are chosen according to the same distributions as the indices sampled in the forward filter and in the backward information filter. It is done in \cite{fearnhead:wyncoll:tawn:2010} by choosing $\adjfunc[smooth]{s}{T}{x,x'} = \adjfuncforward{s}(x)\adjfunc{s}{T}{x'}$ so that \eqref{eq:auxiliary-proposal-smoothing} becomes
\begin{equation}
\label{eq:inst-fearnhead:linear}
\instrpostaux{s}{T}(i,j,x_s) \propto  \ewght{s-1}{i}\adjfuncforward{s}(\epart{s-1}{i}) \frac{\adjfunc{s}{T}{\ebackpart{s+1}{T}{j}} \ebackwght{s+1}{T}{j}}{\gamma_{s+1}(\ebackpart{s+1}{T}{j})} \kiss[smooth]{s}{T}(\epart{s-1}{i},\ebackpart{s+1}{T}{j}; x_s)\eqsp.
\end{equation}
In this case, the importance weight \eqref{eq:weight-fearnhead} associated with each draw $(I_{s}^\ell,\check{I}_{s}^\ell, \smpart{s}{T}{\ell})$ is given by
\begin{equation}
\label{eq:weight-fearnhead:linear}
\smwght{s}{T}{\ell} \eqdef \frac{q(\epart{s-1}{I_{s}^\ell},\smpart{s}{T}{\ell}) g_s(\smpart{s}{T}{\ell}) q(\smpart{s}{T}{\ell},\ebackpart{s+1}{T}{\check{I}_{s}^\ell})}{\adjfuncforward{s}(\epart{s-1}{I_{s}^\ell})\adjfunc{s}{T}{\ebackpart{s+1}{T}{\check{I}_{s}^\ell}}
\kiss[smooth]{s}{T}(\epart{s-1}{I_{s}^\ell},\ebackpart{s+1}{T}{\check{I}_{s}^\ell}; \smpart{s}{T}{\ell}) } \eqsp.
\end{equation}
Instead of sampling new particles at time $s$,  an algorithm similar to the $\bdm$ algorithm of  \cite{briers:doucet:maskell:2010} which uses  the forward particles $\{\epart{s}{\ell}\}_{\ell=1}^N$ or backward particles $\{\ebackpart{s}{T}{\ell}\}_{\ell=1}^N$ may also be implemented with an $\mathcal{O}(N)$ computational complexity. 
\begin{enumerate}[(a)]
\item Choosing $\kiss[smooth]{s}{T}(\epart{s-1}{I_{s}^\ell},\ebackpart{s+1}{T}{\check{I}_{s}^\ell}; x_s) = \kiss{s}{T}(\ebackpart{s+1}{T}{\check{I}_{s}^\ell}, x_s)$ in \eqref{eq:inst-fearnhead:linear}, 
the smoothing distribution approximation is obtained by reweighting the particles obtained in the backward pass. The backward particles $\{\ebackpart{s}{T}{\ell}\}_{\ell=1}^N$ are associated with the importance weights:
\begin{align}
\nonumber
\smwght{s}{T}{\ell} &\eqdef \frac{ \gamma_s(\ebackpart{s}{T}{\ell})g_s(\ebackpart{s}{T}{\ell}) q(\ebackpart{s}{T}{\ell},\ebackpart{s+1}{T}{\check{I}_{s}^\ell})}{\adjfunc{s}{T}{\ebackpart{s+1}{T}{\check{I}_{s}^\ell}}
\kiss{s}{T}(\ebackpart{s+1}{T}{\check{I}_{s}^\ell}, \ebackpart{s}{T}{\ell})} \frac{q(\epart{s-1}{I_{s}^\ell},\ebackpart{s}{T}{\ell})}{\gamma_s(\ebackpart{s}{T}{\ell})\adjfuncforward{s}(\epart{s-1}{I_{s}^\ell})}\eqsp,\\
&= \ebackwght{s}{T}{\ell}\frac{q(\epart{s-1}{I_{s}^\ell},\ebackpart{s}{T}{\ell})}{\gamma_s(\ebackpart{s}{T}{\ell})\adjfuncforward{s}(\epart{s-1}{I_{s}^\ell})}\eqsp.\label{eq:weight-bdm:linear:bkd}
\end{align}
\item Choosing $\kiss[smooth]{s}{T}(\epart{s-1}{I_{s}^\ell},\ebackpart{s+1}{T}{\check{I}_{s}^\ell}; x_s) = \kissforward{s}{s}(\epart{s-1}{I_{s}^\ell}, x_s)$ in \eqref{eq:inst-fearnhead:linear}, the smoothing distribution approximation is obtained by reweighting the particles obtained in the forward filtering pass. The forward particles $\{\epart{s}{\ell}\}_{\ell=1}^N$ are associated with the importance weights:
\begin{equation}
\label{eq:weight-bdm:linear:fwd}
\smwght{s}{T}{\ell} \eqdef \frac{q(\epart{s-1}{I_{s}^\ell},\epart{s}{\ell}) g_s(\epart{s}{\ell}) }{\adjfuncforward{s}(\epart{s-1}{I_{s}^\ell})
\kissforward{s}{s}(\epart{s-1}{I_{s}^\ell}, \epart{s}{\ell}) } \frac{q(\epart{s}{\ell},\ebackpart{s+1}{T}{\check{I}_{s}^\ell})}{\adjfunc{s}{T}{\ebackpart{s+1}{T}{\check{I}_{s}^\ell}}} =  \ewght{s}{\ell}\frac{q(\epart{s}{\ell},\ebackpart{s+1}{T}{\check{I}_{s}^\ell})}{\adjfunc{s}{T}{\ebackpart{s+1}{T}{\check{I}_{s}^\ell}}}\eqsp.
\end{equation}
\end{enumerate}

\section{Exponential deviation inequality for  the two-filter algorithms}
\label{sec:ExponentialTwoFilters}
In this section, we establish exponential deviation inequalities for  the two-filter algorithms introduced in Section~\ref{sec:TwoFilters}. Before stating the results, some additional notations are required. Define, for all $(x,x',x'') \in \Xset^3$,
\begin{equation*}
q^{[2]}(x,x';x'') = q(x,x'') q(x'',x')
\end{equation*}
and for any functions $f: \Xset^2 \to \rset$ and $g: \Xset \to \rset$,
\[
f \odot g(x,x') \eqdef f(x,x') g(x')\eqsp.
\]
Consider the following assumptions:
\begin{assumptionA}
\label{assum:bound-likelihood}
$|q|_{\infty}<\infty$ and for all $0\le t \le T$, $g_t$ is positive and $|g_t|_{\infty}<\infty$.
\end{assumptionA}

\begin{assumptionA}
\label{assum:borne-FFBS}
For all $0\le t \le T$, $|\vartheta_t|_{\infty} < \infty$, $|p_t|_{\infty} < \infty$ and $|\ewght{t}{}|_{\infty} < \infty$ where
    \begin{equation*}
    \omega_0(x) \eqdef \dfrac{\rmd\chi}{\rmd\rho_0}(x)g_{0}(x)
    \quad\mbox{and\;for\;all\;}t\ge 1\quad
    \omega_t(x,x^{\prime}) \eqdef \dfrac{q(x,x^{\prime})g_{t}(x')}{\vartheta_t(x)p_t(x,x^{\prime})}\eqsp.
    \end{equation*}
\end{assumptionA}

\begin{assumptionA}
\label{assum:borne-TwoFilters}
\begin{enumerate}[-]
\item For all $0\le t \le T-1$, $\esssup{\adjfunc{t}{T}{}/\fakeprior_{t+1}}  < \infty$ and $|\kiss{t}{T}|_{\infty}<\infty$. For all $0\le t \le T$ $\esssup{\ebackwghtfunc{t}{T}}  < \infty$, where
\[
\ebackwghtfunc{T}{T}(x)\eqdef \frac{g_T(x)\gamma_T(x)}{\ebackinit_T(x)}\eqsp\mbox{and\;for\;all\;}0\le t< T,\eqsp \ebackwghtfunc{t}{T}(x,x')\eqdef  \frac{\gamma_t(x) g_t(x) q(x,x')}{\adjfunc{t}{T}{x'} \kiss{t}{T}(x',x)}\eqsp.
\]
\item For all $1\le t \le T-1$, $\esssup{\adjfunc[smooth]{t}{T}{} \odot \gamma_{t+1}^{-1}}< \infty$, $\esssup{q \odot \gamma_{t+1}^{-1}}< \infty$, $\esssup{\smwghtfunc{t}{T}}  < \infty$ and $\esssup{\kiss[smooth]{t}{T}}<\infty$ where
\[
\smwght{t}{T}{}(x,x';x'') \eqdef \frac{q^{[2]}(x,x';x'')  g_s(x'') }{\adjfunc[smooth]{t}{T}{x,x''}
\kiss[smooth]{t}{T}(x,x'; x'') } \eqsp.
\]
\end{enumerate}
\end{assumptionA}
We first show that the weighted sample $\{(\swght{s}{i} \ebackwght{t}{T}{j}), (\epartpred{s}{i},\ebackpart{t}{T}{j})\}_{i,j=1}^N$ targets the product distribution $\filt{\Xinit,s}\otimes \backDist[][\gamma]{t}{T}$.
\begin{thm}
\label{thm:Hoeffding-Two-Populations}
Assume that A\ref{assum:bound-likelihood}, A\ref{assum:borne-FFBS} and A\ref{assum:borne-TwoFilters} hold for some $T<\infty$. Then, for all $0 \leq s<t \leq T$, there exist $0 <B_{s,t|T}, C_{s,t|T} <\infty$ such that for  all  $N\ge 1$, $\epsilon > 0$ and all $h \in \mathbb{F}_b(\Xset\times \Xset,\sigmaX\otimes\sigmaX)$,
\begin{equation*}
\PP\left( \left| \sum_{i,j=1}^N \frac{\swght{s}{i}}{\sumwght{s}}\frac{\ebackwght{t}{T}{j}}{\backsumwght{t}{T}} h(\epartpred{s}{i},\ebackpart{t}{T}{j}) - \filt{\Xinit,s}\otimes \backDist[][\gamma]{t}{T}[h] \right| > \epsilon \right) \leq  B_{s,t|T} \rme^{-C_{s,t|T} N \epsilon^2/\oscnorm[2]{h}} \eqsp.
\end{equation*}
\end{thm}
\begin{proof}
The proof is postponed to Section~\ref{proof:thm:Hoeffding-Two-Populations}.
\end{proof}
We now study the weighted sample $\{ (\smwght{s}{T}{i},\smpart{s}{T}{\ell}) \}_{\ell=1}^N$ produced by the $\fwt$ algorithm of Fearnhead, Wyncoll and Tawn \cite{fearnhead:wyncoll:tawn:2010} defined in \eqref{eq:auxiliary-proposal-smoothing} and \eqref{eq:weight-fearnhead} and targeting the marginal smoothing distribution $\post[\chi]{s}{T}$.
\begin{thm}[deviation inequality for $\fwt$ of \cite{fearnhead:wyncoll:tawn:2010}]
\label{thm:Hoeffding-Fearnhead}
Assume that A\ref{assum:bound-likelihood}, A\ref{assum:borne-FFBS} and A\ref{assum:borne-TwoFilters} hold for some $T<\infty$. Then, for all $s<T$, there exist $0 < B_{s|T},C_{s|T} < \infty$ such that for all $N\ge 1$, $\varepsilon>0$ and all  $h \in \functionset[b]{X}$,
\begin{equation*}
\PP\left( \left| \sum_{i=1}^N \frac{\smwght{s}{T}{i}}{\tilde{\Omega}_{s|T}} h(\smpart{s}{T}{i}) - \post[\chi]{s}{T} [h] \right| > \epsilon \right) \leq  B_{s|T} \rme^{-C_{s|T} N \epsilon^2 / \oscnorm[2]{h}} \eqsp.
\end{equation*}
\end{thm}
\begin{proof}
The proof is postponed to Section~\ref{proof:thm:Hoeffding-Fearnhead}.
\end{proof}

Using Theorem~\ref{thm:Hoeffding-Two-Populations} and Lemma~\ref{lem:hoeffding:ratio}, we may derive an exponential inequality for the weighted samples $\{ (\epart{s}{i},\smwght{s}{T}{i, \mathrm{f}}) \}_{i=1}^N$ and $ \{ (\ebackpart{s}{T}{i},\smwght{s}{T}{i, \mathrm{b}}) \}_{i=1}^N$ produced by the $\bdm$ algorithm of  \cite{briers:doucet:maskell:2010}, where $\smwght{s}{T}{i, \mathrm{f}}$  and $\smwght{s}{T}{i, \mathrm{b}}$ are defined in \eqref{eq:definition-smoothing-weight-forward} and \eqref{eq:definition-smoothing-weight-backward}. Therefore, both the forward and the backward particle approximations of the smoothing distribution converge to the marginal smoothing distribution, and these two approximations satisfy an exponential inequality.
\begin{thm}[deviation inequality for the $\bdm$ algorithm of  \cite{briers:doucet:maskell:2010}]
\label{th:exp:forward-backward}
Assume that A\ref{assum:bound-likelihood}, A\ref{assum:borne-FFBS} and A\ref{assum:borne-TwoFilters} hold for some $T<\infty$. Then, for all $1\le s\le T-1$, there exist $0 < B_{s|T},C_{s|T} <\infty$ such that for all $N\ge 1$, $\varepsilon >0$ and all $h \in \functionset[b]{X}$,
\begin{align}
\label{eq:Exponential-Inequality-Forward}
&\PP\left( \left| \sum_{i=1}^N \frac{\smwght{s}{T}{i, \mathrm{f}}}{\tilde{\Omega}^{\mathrm{f}}_{s|T}} h(\epart{s}{i}) - \post[\Xinit]{s}{T} [h] \right| > \epsilon \right) \leq  B_{s|T} \rme^{-C_{s|T} N \epsilon^2 / \oscnorm[2]{h}} \eqsp, \\
\label{eq:Exponential-Inequality-Backward}
&\PP\left( \left| \sum_{i=1}^N \frac{\smwght{s}{T}{i, \mathrm{b}}}{\tilde{\Omega}^{\mathrm{b}}_{s|T}} h(\ebackpart{s}{T}{i}) - \post[\Xinit]{s}{T} [h] \right| > \epsilon \right) \leq  B_{s|T} \rme^{-C_{s|T} N \epsilon^2 / \oscnorm[2]{h}} \eqsp.
\end{align}
\end{thm}
\begin{proof}
The proof is postponed to Section~\ref{proof:th:exp:forward-backward}.
\end{proof}

\begin{rem}
Following \cite{olsson:westerborn:2015,douc:garivier:moulines:olsson:2011,dubarry:lecorff:2013}, time uniform exponential inequalities for the two-filter approximations of the marginal smoothing distributions may be obtained using {\em strong mixing} assumptions which are standard in the SMC literature:

\begin{assumptionmix}
\label{assum:mix}
There exist $0<\sigma_- < \sigma_+<\infty$ and $c_->0$ such that for all $x,x'\in\Xset$, $\sigma_- \le q(x,x')\le \sigma_+$ and for all $t\ge 0$,  \[
\int \Xinit (\rmd x_0)g_0(x_0)\ge c_-\quad\mbox{and}\quad\inf_{x\in\Xset}\int Q(x,\rmd x')g_t(x')\ge c_-\eqsp.
\]
\end{assumptionmix}

\begin{assumptionmix}
\label{assum:mix:gamma}
There exist $0<\gamma_- < \gamma_+<\infty$ and $\check{c}_->0$ such that for all $x\in\Xset$ and all $t\ge0$, $\gamma_- \le \gamma_t(x)\le \gamma_+$ and for all $t\ge 0$,
\[
\int \gamma_T(x_T)g_T(x_T)\rmd x_T\ge \check{c}_- \quad\mbox{and}\quad\inf_{x\in\Xset}\int \gamma_t(x_t)g_t(x_t)q(x_t,x)\gamma_{t+1}^{-1}(x)\rmd x_t\ge \check{c}_-\eqsp.
\]
\end{assumptionmix}

\begin{enumerate}[(i)]
\item \label{it:exp:mix:for}If A\ref{assum:bound-likelihood} and A\ref{assum:borne-FFBS} hold uniformly in $T$ and if H\ref{assum:mix} holds, then, it is proved in \cite{douc:garivier:moulines:olsson:2011} that Proposition~\ref{prop:exponential-inequality-forward} holds with constants that are uniform in time : there exist $0 <B, C <\infty$ such that for all $s \geq 0$, $N > 0$, $\epsilon > 0$ and all $h \in \functionset[b]{X}$,
\begin{equation*}
\PP\left( \left| \sumwght{s}^{-1 } \sum_{i=1}^N \swght{s}{i} h(\epartpred{s}{i}) - \filt{\Xinit,s}[h] \right| \geq \epsilon \right) \leq B \rme^{-C N \epsilon^2 / \oscnorm{h}^2}\eqsp.
\end{equation*}
\item \label{it:exp:mix:back} It can be shown following the exact same steps that if A\ref{assum:bound-likelihood} and A\ref{assum:borne-TwoFilters} hold uniformly in $T$ and if H\ref{assum:mix} and H\ref{assum:mix:gamma} hold then Proposition~\ref{prop:exponential-inequality-backward} holds with constants that are uniform in time: there exist $0<B, C<\infty$  such that for all $t\geq 0$, $N\ge 1$, $\epsilon > 0$, and all $h \in \functionset[b]{X}$,
\[
\PP\left[ \left| \backsumwght{t}{T}^{-1} \sum_{i=1}^N \ebackwght{t}{T}{i} h(\ebackpart{t}{T}{i}) - \backDist[][\gamma]{t}{T}[h] \right| \geq \epsilon \right] \leq B\rme^{-C N \epsilon^2 / \oscnorm{h}^2}  \eqsp.
\]
\item \label{it:exp:mix:twofilt} Therefore, if A\ref{assum:bound-likelihood}, A\ref{assum:borne-FFBS} and A\ref{assum:borne-TwoFilters} hold uniformly in $T$ and if H\ref{assum:mix} and H\ref{assum:mix:gamma} hold, then Theorem~\ref{thm:Hoeffding-Two-Populations} holds with constants that are uniform in time. As a direct consequence, Theorems~\ref{thm:Hoeffding-Fearnhead} and~\ref{th:exp:forward-backward} hold also with constants that are uniform in time.
\end{enumerate}
\end{rem}

\section{Asymptotic normality of the two-filter algorithms}
\label{sec:CLTTwoFilters}
We now establish CLT for the two-filter algorithms. Note first that under assumptions A\ref{assum:bound-likelihood}, A\ref{assum:borne-FFBS} and A\ref{assum:borne-TwoFilters}, for all $0\le s,t\le T$ a CLT may be derived for the weighted samples $\{(\epart{s}{\ell},\ewght{s}{\ell})\}_{\ell=1}^N$ and $\{ (\ebackpart{t}{T}{i},\ebackwght{t}{T}{i}) \}_{i=1}^N$ which target respectively the filtering distribution $\filt{\Xinit,s}$ and the backward information filter $\backDist[][\gamma]{t}{T}$. By Propositions~\ref{prop:CLT-forward} and~\ref{prop:CLT-backward}, there exist $\asymVar[\Xinit]{s}{}$ and $\backasymVar[\gamma]{t}{T}{}$ such that for any $h \in \functionset[b]{X}$,
\begin{align}
N^{1/2} \sum_{i=1}^N \frac{\ewght{s}{i}}{\sumwght{s}} \left(h(\epart{s}{i}) - \filt{\Xinit,s}[h]\right) &\dlim_{N \to \infty} \mathcal{N}\left(0,\asymVar[\Xinit]{s}{h - \filt{\Xinit,s}[h]}\right)\eqsp,\label{eq:CLT:forward}\\
N^{1/2} \sum_{j=1}^N \frac{\ebackwght{t}{T}{j}}{\backsumwght{t}{T}} \left(h(\ebackpart{t}{T}{j})- \backDist[][\gamma]{t}{T}[h]\right) &\dlim_{N \to \infty} \mathcal{N}\left(0,\backasymVar[\gamma]{t}{T}{h-\backDist[][\gamma]{t}{T}[h]}\right)\label{eq:CLT:backward}\eqsp.
\end{align}
Theorem~\ref{thm:CLT-Two-Populations} establishes a CLT for the weighted sample $\{ \swght{s}{i}\ebackwght{t}{T}{j}, (\epartpred{s}{i},\ebackpart{t}{T}{j})\}_{i,j=1}^N$ which targets the product distribution $\filt{\Xinit,s}\otimes \backDist[][\gamma]{t}{T}$. As an important consequence, the asymptotic variance of the weighted sample $\{\swght{s}{i}\ebackwght{t}{T}{j}, (\epartpred{s}{i},\ebackpart{t}{T}{j})\}_{i,j=1}^N$ is the sum of two contributions, the first one involves $\asymVar[\Xinit]{s}{}$ and the second one $\backasymVar[\gamma]{t}{T}{}$. Intuitively, this may be explained by the fact that the estimator $\filt{\Xinit,s}^N\otimes\backDist[][\gamma]{t}{T}^N[h]$ is obtained by mixing two independent weighted samples which suggests the following decomposition:
\begin{multline*}
\sqrt{N}\sum_{i,j=1}^N \frac{\swght{s}{i}}{\sumwght{s}}\frac{\ebackwght{t}{T}{j}}{\backsumwght{t}{T}} \tilde{h}_{s,t}(\epartpred{s}{i},\ebackpart{t}{T}{j}) = \sqrt{N}\sum_{j=1}^N \frac{\ebackwght{t}{T}{j}}{\backsumwght{t}{T}} \filt{\Xinit,s}[\tilde{h}_{s,t}(\cdot,\ebackpart{t}{T}{j})]\\
+ \sqrt{N}\sum_{i=1}^N \frac{\swght{s}{i}}{\sumwght{s}} \backDist[][\gamma]{t}{T}[\tilde{h}_{s,t}(\epartpred{s}{i},\cdot)] + \mathcal{E}^{N}_{s,T|t}(\tilde{h}_{s,t}) \eqsp,
\end{multline*}
where $\tilde{h}_{s,t} = h - \filt{\Xinit,s}\otimes\backDist[][\gamma]{t}{T}[h]$ and
\[
\mathcal{E}^{N}_{s,T|t}(h) \eqdef \sqrt{N}\sum_{i,j=1}^N \frac{\swght{s}{i}}{\sumwght{s}}\frac{\ebackwght{t}{T}{j}}{\backsumwght{t}{T}}\left\{h(\epartpred{s}{i},\ebackpart{t}{T}{j}) -  \filt{\Xinit,s}[h(\cdot,\ebackpart{t}{T}{j})]  - \backDist[][\gamma]{t}{T}[h(\epartpred{s}{i},\cdot)]\right\}\eqsp.
\]
A CLT for the two independent first terms is obtained by \eqref{eq:CLT:forward} and \eqref{eq:CLT:backward}. It remains then to prove that $\mathcal{E}^{N}_{s,T|t}(h)$ converges in probability to $0$. However, this cannot be obtained directly from the exponential deviation inequality derived in Theorem~\ref{thm:Hoeffding-Two-Populations} and requires sharper controls of the smoothing error (for instance nonasymptotic $\mathrm{L}^p$-mean error bounds).
Theorem~\ref{thm:CLT-Two-Populations}  provides a direct proof following the asymptotic theory of weighted system of particles developed in \cite{douc:moulines:2008}.
\begin{thm}
\label{thm:CLT-Two-Populations}
Assume that A\ref{assum:bound-likelihood}, A\ref{assum:borne-FFBS} and A\ref{assum:borne-TwoFilters} hold for some $T<\infty$. Then,  for all $0 \leq s<t \leq T$ and all $h \in \functionset[b]{X\times X}$,
\begin{multline*}
\sqrt{N}\left(\sum_{i,j=1}^N \frac{\swght{s}{i}}{\sumwght{s}}\frac{\ebackwght{t}{T}{j}}{\backsumwght{t}{T}} h(\epartpred{s}{i},\ebackpart{t}{T}{j}) - \filt{\Xinit,s}\otimes \backDist[][\gamma]{t}{T}[h]
\right) \\
\dlim_{N \to \infty} \mathcal{N}\left(0,\asymVarJoint{s,t}{T}{h- \filt{\Xinit,s}\otimes \backDist[][\gamma]{t}{T}[h]}\right) \eqsp,
\end{multline*}
where $\asymVarJoint{s,t}{T}{h}$ is defined by:
\begin{equation}
\label{eq:asymVarjoint}
\asymVarJoint{s,t}{T}{h} \eqdef \asymVar[\Xinit]{s}{\int \backDist[][\gamma]{t}{T}(\rmd x_t)h(\cdot,x_t)} +  \backasymVar[\gamma]{t}{T}{\int \filt{\Xinit,s}(\rmd x_{s}) h(x_s,\cdot)}\eqsp,
\end{equation}
with $\asymVar[\Xinit]{s}{}$ and $\backasymVar[\gamma]{t}{T}{}$ are given in Proposition~\ref{prop:CLT-forward} and Proposition~\ref{prop:CLT-backward}.
\end{thm}

\begin{proof}
The proof is postponed to Section~\ref{proof:thm:CLT-Two-Populations}.
\end{proof}
Define
\begin{align*}
\sigma_s &\eqdef \filt{\Xinit,s-1} \otimes \backDist[][\gamma]{s+1}{T} \left[ \int  q^{[2]}(\cdot,x) g_s(x) \rmd x \odot \fakeprior_{s+1}^{-1} \right] \eqsp, \\
\Sigma_s[h] &\eqdef \asymVarJoint{s-1,s+1}{T}{\int q^{[2]}(\cdot;x) g_s(x)  h(x) \rmd x \odot \fakeprior_{s+1}^{-1}} \eqsp.
\end{align*}
Theorem~\ref{thm:CLT-Fearnhead-smoother} provides a CLT for the $\fwt$ algorithm of \cite{fearnhead:wyncoll:tawn:2010}
\begin{thm}[CLT for the $\fwt$ algorithm of \cite{fearnhead:wyncoll:tawn:2010}]
\label{thm:CLT-Fearnhead-smoother}
Assume that A\ref{assum:bound-likelihood}, A\ref{assum:borne-FFBS} and A\ref{assum:borne-TwoFilters} hold for some $T<\infty$. Then,  for all $1\le s \le T-1$ and all $h \in \functionset[b]{X}$,
\begin{equation*}
\sqrt{N}\left( \sum_{i=1}^N \frac{\smwght{s}{T}{i}}{\tilde{\Omega}_{s|T}} h(\smpart{s}{T}{i}) - \post[\Xinit]{s}{T}[h]
\right)
\dlim_{N \to \infty} \mathcal{N}\left(0,\asymVarFearnhead[\Xinit]{s}{T}{h - \post[\Xinit]{s}{T}[h]}\right) \eqsp.
\end{equation*}
where
\begin{multline}
\label{eqdef:asymvarfearnhead} 
\asymVarFearnhead[\Xinit]{s}{T}{h}= \sigma_s^{-2} \left\{ \Sigma_s[h] + \filt{\Xinit,s-1} \otimes \backDist[][\gamma]{s+1}{T} \left[\adjfunc[smooth]{s}{T}{} \odot \fakeprior^{-1}_{s+1} \right] \phantom{\int} \right. \\  \times \left.  \filt{\Xinit,s-1} \otimes \backDist[][\gamma]{s+1}{T}\left[ \int \smwghtfunc{s}{T}(\cdot;x) q^{[2]}(\cdot,x) g_s(x) h^2(x) \rmd x \odot \fakeprior^{-1}_{s+1} \right]\right\} \eqsp.
\end{multline}
\end{thm}

\begin{proof}
The proof is postponed to Section~\ref{proof:thm:CLT-Fearnhead-smoother}.
\end{proof}

The decompositions \eqref{eq:forward-approximation} and \eqref{eq:backward-approximation} together with Theorem \ref{thm:CLT-Two-Populations} allow to prove a CLT form the forward and the backward
approximations of the marginal smoothing distribution. Theorem~\ref{thm:CLT-Doucet-Forward-Backward-Smoother} is a direct consequence of Proposition~\ref{prop:CLT-backward}, Theorem~\ref{thm:CLT-Two-Populations} and Slutsky Lemma.
\begin{thm}[CLT for the $\bdm$ algorithm of  \cite{briers:doucet:maskell:2010}]
\label{thm:CLT-Doucet-Forward-Backward-Smoother}
Assume that A\ref{assum:bound-likelihood}, A\ref{assum:borne-FFBS} and A\ref{assum:borne-TwoFilters} hold for some $T<\infty$. Then, for all $1 \leq s \leq T-1$ and all $h \in \functionset[b]{X}$,
\begin{equation*}
\sqrt{N}\left( \sum_{i=1}^N \frac{\smwght{s}{T}{i, \mathrm{f}}}{\tilde{\Omega}^{\mathrm{f}}_{s|T}} h(\epart{s}{i}) - \post[\Xinit]{s}{T}[h]
\right) \dlim_{N \to \infty} \mathcal{N}\left(0,\asymVarDoucet[\Xinit]{\forward}{s}{T}{h- \post[\Xinit]{s}{T}[h]} \right)\eqsp,
\end{equation*}
where
\begin{align*}
\asymVarDoucet[\Xinit]{\forward}{s}{T}{h} &\eqdef \asymVarJoint{s,s+1}{T}{H^{\forward}_s} /\{ \filt{\Xinit,s} \otimes \backDist[][\gamma]{s+1}{T}[q \odot \fakeprior_{s+1}^{-1}] \}^2\eqsp,\\
H^{\forward}_s (x,x') &\eqdef h(x) q(x,x') \fakeprior_{s+1}^{-1}(x')\eqsp.
\end{align*}
Similarly,
\begin{equation*}
\sqrt{N}\left( \sum_{i=1}^N \frac{\smwght{s}{T}{i, \mathrm{b}}}{\tilde{\Omega}^{\mathrm{b}}_{s|T}} h(\epart{s}{i}) - \post[\Xinit]{s}{T}[h]
\right)
\dlim_{N \to \infty} \mathcal{N}\left(0, \asymVarDoucet[\Xinit]{\backward}{s}{T}{h- \post[\Xinit]{s}{T}[h]}\right)\eqsp,
\end{equation*}
where
\begin{align*}
\asymVarDoucet[\Xinit]{\backward}{s}{T}{h} &\eqdef \asymVarJoint{s-1,s}{T}{H_s^{\backward} }/\{ \filt{\Xinit,s-1} \otimes \backDist[][\gamma]{s}{T}[q \odot \fakeprior_s^{-1}] \}^2\eqsp,\\
H_s^{\backward} (x,x') &\eqdef  q(x,x') \fakeprior_s^{-1}(x')h(x')\eqsp.
\end{align*}
\end{thm}
Note that $\sigma_s$ and $\Sigma_s[h]$ may be written as:
\[
\sigma_s = \filt{\Xinit,s} \otimes \backDist[][\gamma]{s+1}{T} \left[q\odot \fakeprior_{s+1}^{-1} \right]\times \filt{\Xinit,s-1}\left[ \int  q(\cdot,x) g_s(x) \rmd x\right]
\]
and by Theorem~\ref{thm:CLT-Two-Populations},
\begin{multline*}
\Sigma_s[h] = \asymVar[\Xinit]{s-1}{\int q(\cdot,x)g_s(x)h^1_{s+1}(x)\rmd x} \\
+ \filt{\Xinit,s-1}^2\left[ \int  q(\cdot,x) g_s(x) \rmd x\right]\backasymVar[\gamma]{s+1}{T}{h^2_{s+1}}\eqsp,
\end{multline*}
with $h^1_{s+1}(x)\eqdef h(x)\backDist[][\gamma]{s+1}{T}[q(x,\cdot)\gamma_{s+1}^{-1}]$ and $h^2_{s+1}(x)\eqdef \gamma_{s+1}^{-1}(x)\filt{\Xinit,s}[h(\cdot)q(\cdot,x)]$. In the case where $\kiss[smooth]{s}{T}(x_s,x_{s+1}; x_s) = \kissforward{s}{s}(x_{s-1}, x_s)$ in \eqref{eq:inst-fearnhead:linear} and $\adjfunc[smooth]{s}{T}{x,x'} = \adjfuncforward{s}(x)\adjfunc{s}{T}{x'}$, the smoothing distribution approximation given by the $\fwt$ algorithm is obtained by reweighting the particles obtained in the forward filtering pass and $\asymVarFearnhead[\Xinit]{s}{T}{h}$ may be compared to $\asymVarDoucet[\Xinit]{\forward}{s}{T}{h}$ as both approximations of $\post[\Xinit]{s}{T}[h]$ are based on the same particles (associated with different importance weights). In this case, the two last terms in \eqref{eqdef:asymvarfearnhead} are easily interpreted in the case  $\adjfunc{s}{T}{} = \gamma_{s+1}$:
\[
\filt{\Xinit,s-1} \otimes \backDist[][\gamma]{s+1}{T} \left[\adjfunc[smooth]{s}{T}{} \odot \fakeprior^{-1}_{s+1} \right] = \filt{\Xinit,s-1}[ \adjfuncforward{s}]\backDist[][\gamma]{s+1}{T}[\adjfunc{s}{T}{}\fakeprior^{-1}_{s+1}] = \filt{\Xinit,s-1}[ \adjfuncforward{s}]
\]
and by Jensen's inequality,
\begin{align*}
&\filt{\Xinit,s-1} \otimes \backDist[][\gamma]{s+1}{T}\left[ \int \smwghtfunc{s}{T}(\cdot;x) q^{[2]}(\cdot,x) g_s(x) h^2(x) \rmd x \odot \fakeprior^{-1}_{s+1} \right] \\
&\hspace{.5cm}= \int \filt{\Xinit,s-1}(\rmd x_{s-1})\omega_s(x_{s-1},x)g_s(x)q(x_{s-1},x)\backDist[][\gamma]{s+1}{T}[q^2(x,\cdot)\gamma_{s+1}^{-2}]h^2(x) \rmd x\eqsp,\\
&\hspace{.5cm}\ge \int \filt{\Xinit,s-1}(\rmd x_{s-1})\omega_s(x_{s-1},x)g_s(x)q(x_{s-1},x)(h^1_{s+1}(x))^2 \rmd x\eqsp.
\end{align*}
Therefore, by Proposition~\ref{prop:CLT-backward} and Theorem~\ref{thm:CLT-Doucet-Forward-Backward-Smoother}
\[
\asymVarFearnhead[\Xinit]{s}{T}{h} \ge \frac{\asymVar[\Xinit]{s}{h^1_{s+1}} + \backasymVar[\gamma]{s+1}{T}{h^2_{s+1}}}{\left(\filt{\Xinit,s} \otimes \backDist[][\gamma]{s+1}{T} \left[q\odot \fakeprior_{s+1}^{-1} \right]\right)^2}=\asymVarDoucet[\Xinit]{\forward}{s}{T}{h}\eqsp,
\]
where the last inequality comes from Theorem~\ref{thm:CLT-Two-Populations}. The same inequality holds for $\asymVarDoucet[\Xinit]{\backward}{s}{T}{h}$ when $\kiss[smooth]{s}{T}(x_{s-1},x_{s+1}; x_s) = \kiss{s}{T}(x_{s+1}, x_s)$ in \eqref{eq:inst-fearnhead:linear}.

\begin{rem}
Under the strong mixing assumptions H\ref{assum:mix} and H\ref{assum:mix:gamma}, time uniform bounds for the asymptotic variances of  the two-filter approximations of the marginal smoothing distributions may be obtained.
\begin{enumerate}[(i)]
\item If A\ref{assum:bound-likelihood} and A\ref{assum:borne-FFBS} hold uniformly in $T$ and if H\ref{assum:mix} holds, then it is proved in \cite{douc:garivier:moulines:olsson:2011} that there exists $C>0$ such that for all $s \geq 0$ and all $h \in \functionset[b]{X}$, the asymptotic variance $\asymVar[\Xinit]{s}{h}$ defined in Proposition~\ref{prop:CLT-forward} satisfies:
\[
\asymVar[\Xinit]{s}{h} \le C \esssup{h}^2\eqsp.
\]
\item Following the same steps, if A\ref{assum:bound-likelihood} and A\ref{assum:borne-TwoFilters} hold uniformly in $T$  and if H\ref{assum:mix} and H\ref{assum:mix:gamma} hold, there exists $C>0$ such that for all $0\le t\le T$ and all $h \in \functionset[b]{X}$, the asymptotic variance $\backasymVar{t}{T}{h}$ defined in Proposition~\ref{prop:CLT-backward} satisfies:
\[
\backasymVar[\gamma]{t}{T}{h}\le C \esssup{h}^2\eqsp.
\]
\item As a consequence, if A\ref{assum:bound-likelihood}, A\ref{assum:borne-FFBS} and A\ref{assum:borne-TwoFilters} hold uniformly in $T$  and  if H\ref{assum:mix} and H\ref{assum:mix:gamma} hold, the asymptotic variances $\asymVarJoint{s,t}{T}{h}$, $\asymVarDoucet[\Xinit]{\forward}{s}{T}{h}$, $\asymVarDoucet[\Xinit]{\backward}{s}{T}{h}$ and $\asymVarFearnhead[\Xinit]{s}{T}{h}$ defined in Theorem~\ref{thm:CLT-Two-Populations}, Theorem~\ref{thm:CLT-Fearnhead-smoother}  and Theorem~\ref{thm:CLT-Doucet-Forward-Backward-Smoother} are all uniformly bounded.
\end{enumerate}
\end{rem} 

\section{Proofs}
\label{sec:proofs}
\subsection{Proof of Theorem~\ref{thm:Hoeffding-Two-Populations}}
\label{proof:thm:Hoeffding-Two-Populations}
Define $\mathcal{G}^N_{t|T} \eqdef \sigma(\ebackpart{t}{T}{j},\ebackwght{t}{T}{j}, 1\leq j \leq N)$ and
$$
f_{t|T}(x) \eqdef  \backsumwght{t}{T}^{-1}\sum_{j=1}^N\ebackwght{t}{T}{j}h(x,\ebackpart{t}{T}{j})
$$
whose oscillation is bounded by $\oscnorm{h}$. By the exponential inequality for the  auxiliary particle filter (Proposition~\ref{prop:exponential-inequality-forward}), there exist constants $B_{s}$ and $C_{s}$ such that
\begin{multline}\label{eq:inegExpoBack2}
\PP\left( \left| \sum_{i,j=1}^N \frac{\swght{s}{i}}{\sumwght{s}}\frac{\ebackwght{t}{T}{j}}{\backsumwght{t}{T}} h(\epartpred{s}{i},\ebackpart{t}{T}{j}) -\sum_{j=1}^N \frac{\ebackwght{t}{T}{j}}{\backsumwght{t}{T}} \int \filt{\Xinit,s}(\rmd x_s)h(x_s,\ebackpart{t}{T}{j})   \right| > \epsilon \right)\\
=\PE\left[ \CPP{\left|\sum_{i=1}^N \frac{\swght{s}{i}}{\sumwght{s}} f_{t|T}(\epartpred{s}{i})-\filt{\Xinit,s}(f_{t|T})\right|>\epsilon}{\mathcal{G}^N_{t|T}}\right]\leq  B_{s} \rme^{-C_{s} N \epsilon^2 / \oscnorm[2]{h}}\eqsp.
\end{multline}
Since the oscillation of the function $x \mapsto \int  \filt{\Xinit,s}(\rmd x_s)h(x_s,x)$ is bounded by $\oscnorm{h}$, by Proposition~\ref{prop:exponential-inequality-backward} there exist constants $B_{t|T}$ and $C_{t|T}$
such that
\begin{multline}
\label{eq:inegExpoBack1}
\PP\left(  \left| \sum_{j=1}^N \frac{\ebackwght{t}{T}{j}}{\backsumwght{t}{T}} \int \filt{\Xinit,s}(\rmd x_s)h(x_s,\ebackpart{t}{T}{j})  - \filt{\Xinit,s}\otimes \backDist[][\gamma]{t}{T}[h] \right|> \epsilon \right)\\
\leq  B_{t|T} \rme^{-C_{t|T} N \epsilon^2/\oscnorm[2]{h}}\eqsp,
\end{multline}
which concludes the proof.
\subsection{Proof of Theorem~\ref{thm:Hoeffding-Fearnhead}}
\label{proof:thm:Hoeffding-Fearnhead}
Define $\tilde{h}_{s|T} \eqdef h - \post[\chi]{s}{T} [h]$. Lemma~\ref{lem:hoeffding:ratio} is used with
\begin{align*}
a_N &\eqdef N^{-1} \sum_{i=1}^N \smwght{s}{T}{i} \tilde{h}_{s|T}(\smpart{s}{T}{i})\eqsp, \quad b_N \eqdef N^{-1} \tilde{\Omega}_{s|T}\eqsp,\\
b& \eqdef \frac{\filt{\Xinit,s}\otimes \backDist[][\gamma]{s+1}{T}\left[\int q^{[2]}(\cdot;x_s) g_s(x_s) \rmd x_s\odot \gamma_{s+1}^{-1}\right]}{\filt{\Xinit,s}\otimes \backDist[][\gamma]{s+1}{T}\left[\adjfunc[smooth]{s}{T}{} \odot \gamma_{s+1}^{-1}\right]}\eqsp.
\end{align*}
Lemma~\ref{lem:hoeffding:ratio}-\eqref{it:bound:ab} is satisfied using $\beta \eqdef b$ and  $|a_N|/|b_N| \le \oscnorm{h}$. To prove Lemma~\ref{lem:hoeffding:ratio}-\eqref{it:hoeff:a} for $a_N$, note that Hoeffding inequality implies that, for any $\epsilon > 0$,
\[
\CPP{\left| a_N - \CPE{\smwght{s}{T}{1} \tilde{h}_{s|T}(\smpart{s}{T}{1})}{\mathcal{G}^N_{s,T}} \right| \geq \epsilon}{\mathcal{G}^N_{s,T}} \leq 2\, \mathrm{exp}\left\{-\frac{N\epsilon^2}{8\esssup{\smwghtfunc{s}{T}}^2\oscnorm[2]{h}}\right\} \eqsp,
\]
where $\mathcal{G}^N_{s,T} \eqdef \mathcal{G}^{N,+}_{s-1} \vee \mathcal{G}^{N,-}_{s+1,T}$ and
\begin{align*}
\mathcal{G}^{N,+}_{s}&\eqdef \sigma\left\{ \{ (\ewght{u}{i}, \epart{u}{i}) \}_{i=1}^N, u =1, \dots, s-1\right\}\eqsp,\\
\mathcal{G}^{N,-}_{s,T} &\eqdef \sigma \left\{ \{ (\ebackwght{u}{T}{i},\ebackpart{u}{T}{i}) \}_{i=1}^N, u=s+1,\dots, T \right\}\eqsp.
\end{align*}
On the other hand, for all $\ell \in\{1,\ldots,N\}$,
\begin{multline*}
\CPE{\smwght{s}{T}{\ell} \tilde{h}_{s|T}(\smpart{s}{T}{\ell})}{\mathcal{G}^N_{s,T}} \\
= \frac{\sum_{i,j=1}^N \ewght{s-1}{i} \ebackwght{s+1}{T}{j}\gamma_{s+1}^{-1}(\ebackpart{s+1}{T}{j}) \int q^{[2]}(\epart{s-1}{i},\ebackpart{s+1}{T}{j};x_s) g_s(x_s) \tilde{h}_{s|T}(x_s) \rmd x_s}{\sum_{i,j=1}^N \ewght{s-1}{i} \ebackwght{s+1}{T}{j}\gamma_{s+1}^{-1}(\ebackpart{s+1}{T}{j})\adjfunc[smooth]{s}{T}{\epart{s-1}{i},\ebackpart{s+1}{T}{j}}} \eqsp.
\end{multline*}
The proof of Lemma~\ref{lem:hoeffding:ratio}-\eqref{it:hoeff:a} is then completed by applying Lemma~\ref{lem:hoeffding:ratio} to $a'_N$, $b'_N$ and $b'$ defined by:
\begin{align*}
a'_N & \eqdef \sum_{i,j=1}^N \frac{\ewght{s-1}{i}\,\ebackwght{s+1}{T}{j}}{\sumwght{s-1}\,\backsumwght{s+1}{T}}\gamma_{s+1}^{-1}(\ebackpart{s+1}{T}{j}) \int q^{[2]}(\epart{s-1}{i},\ebackpart{s+1}{T}{j};x_s) g_s(x_s) \tilde{h}_{s|T}(x_s) \rmd x_s\eqsp,\\
b'_N & \eqdef \sum_{i,j=1}^N \frac{\ewght{s-1}{i}\,\ebackwght{s+1}{T}{j}}{\sumwght{s-1}\,\backsumwght{s+1}{T}}\gamma_{s+1}^{-1}(\ebackpart{s+1}{T}{j})\adjfunc[smooth]{s}{T}{\epart{s-1}{i},\ebackpart{s+1}{T}{j}}\eqsp,\\
b' &\eqdef \filt{\Xinit,s}\otimes \backDist[][\gamma]{s+1}{T}[\adjfunc[smooth]{s}{T}{} \odot \gamma_{s+1}^{-1}]\eqsp.
\end{align*}
Note first that Lemma~\ref{lem:hoeffding:ratio}-\eqref{it:bound:ab} is satisfied using $\beta' \eqdef b'$ and $\left|a'_N/b'_N\right|\le \esssup{\smwghtfunc{s}{T}}\oscnorm{h}$. In addition, by \eqref{eq:filtBackfilt},
\[
\filt{\Xinit,s}\otimes \backDist[][\gamma]{s+1}{T}[\bar{h}_{s|T}] \propto \post[\chi]{s}{T}[\tilde{h}_{s|T}] = 0\eqsp,
\]
where
\[
\bar{h}_{s|T}(x,x') \eqdef \int q^{[2]}(\cdot;x_s) g_s(x_s)\tilde{h}_{s|T}(x_s) \rmd x_s\odot\gamma_{s+1}^{-1}(x,x')\eqsp.
\]
Theorem~\ref{thm:Hoeffding-Two-Populations} ensures that Lemma~\ref{lem:hoeffding:ratio}-\eqref{it:hoeff:a} is satisfied for $a'_N$
as
$$
\oscnorm{\bar{h}_{s|T}}\le 2 \esssup{\adjfunc[smooth]{s}{T}{} \odot \gamma_{s+1}^{-1}}\esssup{\smwghtfunc{s}{T}}\oscnorm{h} \eqsp.
$$
Similarly, Theorem~\ref{thm:Hoeffding-Two-Populations} yields:
\[
\mathbb{P}\left(\left|b'_N - b' \right|\ge \epsilon\right)\le B_s \rme^{-C_s N \epsilon^2 /\oscnorm[2]{\adjfunc[smooth]{s}{T}{} \odot \gamma_{s+1}^{-1}}}\eqsp,
\]
which proves  Lemma~\ref{lem:hoeffding:ratio}-\eqref{it:hoeff:b} for $b'_N$ and concludes the proof of Lemma~\ref{lem:hoeffding:ratio}-\eqref{it:hoeff:a} for $a_N$. The proof of Lemma~\ref{lem:hoeffding:ratio}-\eqref{it:hoeff:b} for $b_N$ is along the same lines.

\subsection{Proof of Theorem~\ref{th:exp:forward-backward}}
\label{proof:th:exp:forward-backward}
Define
\[
\fup{h}{s}(x,x')\eqdef \fakeprior^{-1}_{s+1}(x') h(x) q(x,x')\quad\mbox{and}\quad\fdown{h}{s}(x,x')\eqdef \fakeprior_s^{-1}(x') q(x,x') h(x')\eqsp.
\]
It follows from the definition of the forward and backward smoothing weights \eqref{eq:definition-smoothing-weight-forward} and \eqref{eq:definition-smoothing-weight-backward} that,
\begin{align}
\label{eq:forward-approximation}
& \sum_{i=1}^N \frac{\smwght{s}{T}{i, \mathrm{f}}}{\tilde{\Omega}^{\mathrm{f}}_{s|T}} h(\epart{s}{i}) =
\frac{\sumwght{s}^{-1} \backsumwght{s+1}{T}^{-1} \sum_{i,j=1}^N \ewght{s}{i} \ebackwght{s+1}{T}{j} \fup{h}{s}(\epart{s}{i},\ebackpart{s+1}{T}{j}) }
{\sumwght{s}^{-1} \backsumwght{s+1}{T}^{-1} \sum_{i,j=1}^N \ewght{s}{i} \ebackwght{s+1}{T}{j}  \fup{\1}{s}(\epart{s}{i},\ebackpart{s+1}{T}{j}) } \eqsp, \\
\label{eq:backward-approximation}
& \sum_{i=1}^N \frac{\smwght{s}{T}{i, \mathrm{b}}}{\tilde{\Omega}^{\mathrm{b}}_{s|T}} h(\ebackpart{s}{T}{i})= \frac{\sumwght{s-1}^{-1} \backsumwght{s}{T}^{-1} \sum_{i,j=1}^N \ewght{s-1}{i} \ebackwght{s}{T}{j} \fdown{h}{s}(\epart{s-1}{i},\ebackpart{s}{T}{j})}{\sumwght{s-1}^{-1} \backsumwght{s}{T}^{-1} \sum_{i,j=1}^N \ewght{s-1}{i} \ebackwght{s}{T}{j}\fdown{\1}{s}(\epart{s-1}{i},\ebackpart{s}{T}{j})} \eqsp.
 \end{align}
 On the other hand, from the definition of the filtering distribution and of the backward information filter
\begin{align*}
& \post[\Xinit]{s}{T}[h] = \filt{\Xinit,s}\otimes\backDist{s+1}{T}\left[\fup{h}{s}\right]/\filt{\Xinit,s}\otimes\backDist{s+1}{T}\left[\fup{\1}{s}\right] \eqsp, \\
& \post[\Xinit]{s}{T}[h] = \filt{\Xinit,s-1}\otimes\backDist{s}{T}\left[\fdown{h}{s}\right]/\filt{\Xinit,s-1}\otimes\backDist{s}{T}\left[\fdown{\1}{s}\right]\eqsp.
\end{align*}
Then, \eqref{eq:Exponential-Inequality-Forward} is established by writing:
\[
\sum_{i=1}^N \frac{\smwght{s}{T}{i, \mathrm{f}}}{\tilde{\Omega}^{\mathrm{f}}_{s|T}} h(\epart{s}{i}) - \post[\Xinit]{s}{T} [h] = a_N^{i, \mathrm{f}}/b_N^{i, \mathrm{f}}\eqsp,
\]
where
\begin{align*}
a_N^{i, \mathrm{f}}&\eqdef \sum_{i,j=1}^N \frac{\ewght{s}{i} \ebackwght{s+1}{T}{j}}{\sumwght{s}\backsumwght{s+1}{T}}\fup{\1}{s}(\epart{s}{i},\ebackpart{s+1}{T}{j})\left\{\frac{\fup{h}{s}(\epart{s}{i},\ebackpart{s+1}{T}{j})}{\fup{\1}{s}(\epart{s}{i},\ebackpart{s+1}{T}{j})}-\frac{\filt{\Xinit,s}\otimes\backDist{s+1}{T}\left[\fup{h}{s}\right]}{\filt{\Xinit,s}\otimes\backDist{s+1}{T}\left[\fup{\1}{s}\right]}\right\}\eqsp,\\
b_N^{i, \mathrm{f}} & \eqdef \sum_{i,j=1}^N \frac{\ewght{s}{i} \ebackwght{s+1}{T}{j}}{\sumwght{s}\backsumwght{s+1}{T}}  \fup{\1}{s}(\epart{s}{i},\ebackpart{s+1}{T}{j})\eqsp,\quad b\eqdef \filt{\Xinit,s}\otimes\backDist{s+1}{T}\left[\fup{\1}{s}\right]\eqsp.
\end{align*}
Lemma~\ref{lem:hoeffding:ratio} may then be applied with $\beta \eqdef b$. Note that
\[
\frac{\fup{h}{s}(\epart{s}{i},\ebackpart{s+1}{T}{j})}{\fup{\1}{s}(\epart{s}{i},\ebackpart{s+1}{T}{j})}-\frac{\filt{\Xinit,s}\otimes\backDist{s+1}{T}\left[\fup{h}{s}\right]}{\filt{\Xinit,s}\otimes\backDist{s+1}{T}\left[\fup{\1}{s}\right]} =  h(\epart{s}{i}) -  \post[\Xinit]{s}{T}[h]\eqsp,
\]
which ensures that $\left |a_N^{i, \mathrm{f}}/b_N^{i, \mathrm{f}}\right|\le \oscnorm{h}$ and that Lemma~\ref{lem:hoeffding:ratio}-\eqref{it:bound:ab} is satisfied. By
\begin{align*}
&\oscnorm{\fup{\1}{s}} = \oscnorm{ q\odot \gamma_{s+1}^{-1}}\eqsp,\\
&\oscnorm{\fup{\1}{s}\odot\left\{h(\epart{s}{i}) -  \post[\Xinit]{s}{T}[h]\right\}} \le 2 \esssup{q\odot \gamma_{s+1}^{-1}}\oscnorm{h}\eqsp,
\end{align*}
Theorem~\ref{thm:Hoeffding-Two-Populations} shows that Lemma~\ref{lem:hoeffding:ratio}-\eqref{it:hoeff:a} and \eqref{it:hoeff:b} are satisfied. The proof of \eqref{eq:Exponential-Inequality-Backward} follows the exact same lines.

\subsection{Proof of Theorem~\ref{thm:CLT-Two-Populations}}
\label{proof:thm:CLT-Two-Populations}
For all $1\le t\le T$, the result is shown by induction on $s$ where $s \in \{0, \dots, t-1\}$. Write $\tilde{h}_{0,t}\eqdef h - \filt{\Xinit,0}\otimes \backDist[][\gamma]{t}{T}[h]$ and set, for $i \in \{1,\dots,N\}$,
\begin{equation*}
U_{N,i} \eqdef N^{-1/2} \ewght{0}{i} \sum_{j=1}^N \frac{\ebackwght{t}{T}{j}}{\backsumwght{t}{T}} \tilde{h}_{0,t}(\epart{0}{i},\ebackpart{t}{T}{j})\eqsp.
\end{equation*}
Then,
\[
\sqrt{N}\left(\sum_{i,j=1}^N \frac{\swght{0}{i}}{\sumwght{0}}\frac{\ebackwght{t}{T}{j}}{\backsumwght{t}{T}} h(\epartpred{0}{i},\ebackpart{t}{T}{j}) - \filt{\Xinit,0}\otimes \backDist[][\gamma]{t}{T}[h]
\right) = \left(\sumwght{0}/N\right)^{-1}\sum_{i=1}^N U_{N,i}\eqsp.
\]
Define $\mathcal{G}_{N,i} \eqdef  \sigma\left(\{\epart{0}{\ell}\}_{\ell \leq i}, \{\ebackpart{u}{T}{j}\}_{t \leq u \leq T}, j=1,\dots,N \right)$. Then,
\[
\sum_{i=1}^N \CPE{U_{N,i}}{\mcg{N,i-1}} = N^{1/2} \sum_{j=1}^N \frac{\ebackwght{t}{T}{j}}{\backsumwght{t}{T}} \XinitIS{0}\left[\ewght{0}{}\tilde{h}_{0,t}(\cdot, \ebackpart{t}{T}{j}) \right]\eqsp.
\]
As $\int \backDist[][\gamma]{t}{T}(\rmd x_t)\XinitIS{0}(\rmd x_0)\ewght{0}{}(x_0)\tilde{h}_{0,t}(x_0,x_t) = 0$, by the CLT for the backward information filter (Proposition~\ref{prop:CLT-backward}),
\begin{equation*}
\label{eq:CLT-1}
\sum_{i=1}^N \CPE{U_{N,i}}{\mcg{N,i-1}} \dlim_{N\to\infty} \mathcal{N}\left( 0, \backasymVar{t}{T}{H_{0,t}} \right) \eqsp,
\end{equation*}
where $H_{0,t}(x_t)\eqdef \int \XinitIS{0}(\rmd x_0)\ewght{0}{}(x_0)\tilde{h}_{0,t}(x_0,x_t)$.
We now prove that
\[
\CPE{\exp\left(\rmi u \sum_{i=1}^N \{U_{N,i}-\CPE{U_{N,i}}{\mcg{N,i-1}}\} \right)}{\mcg{N,0}}\plim_{N \to \infty} \exp\left(-\frac{u^2 \sigma^2_{0,t|T}[h]}{2}\right)\eqsp,
\]
where
\[
\sigma^2_{0,t|T}[h]\eqdef  \int \XinitIS{0}(\rmd x)\ewght{0}{2}(x)\backDist[][\gamma]{t}{T}^2[\tilde{h}_{0,t}(x,\cdot)]\eqsp.
\]
This is done by applying \cite[Theorem~A.3]{douc:moulines:2008} which requires to show that
\begin{align}
\label{eq:condRD-1}
 &\sum_{i=1}^N \left(\CPE{U_{N,i}^2}{\mcg{N,i-1}}-\CPE{U_{N,i}}{\mcg{N,i-1}}^2 \right)\plim_{N \to \infty} \sigma^2_{0,t|T}[h]\eqsp, \\
\label{eq:condRD-3}
 &\sum_{i=1}^N \CPE{U_{N,i}^2 \1\{|U_{N,i}|>\varepsilon\}}{\mcg{N,i-1}} \plim_{N \to \infty} 0\eqsp.
\end{align}
By Proposition~\ref{prop:exponential-inequality-backward},
\[
\sum_{i=1}^N \CPE{U_{N,i}}{\mcg{N,i-1}}^2 = \left( \sum_{j=1}^N \frac{\ebackwght{t}{T}{j}}{\backsumwght{t}{T}} H_{0,t}(\ebackpart{t}{T}{j}) \right)^2  \plim_{N\to\infty} \backDist[][\gamma]{t}{T}^2[H_{0,t}] = 0\eqsp.
\]
On the other hand,
\begin{align*}
\mathbb{E}\left[\left|\sum_{i=1}^N \CPE{U_{N,i}^2}{\mcg{N,i-1}} -\sigma^2_{0,t|T}[h]\right|\right]& \\
&\hspace{-5cm}= \int \XinitIS{0}(\rmd x)\ewght{0}{2}(x) \mathbb{E}\left[\left|\left( \sum_{j=1}^N \frac{\ebackwght{t}{T}{j}}{\backsumwght{t}{T}} \tilde{h}_{0,t}(x,\ebackpart{t}{T}{j}) \right)^2-\backDist[][\gamma]{t}{T}^2[ \tilde{h}_{0,t}(x,\cdot)]\right|\right]\eqsp,\\
&\hspace{-5cm}\le 2\oscnorm{h}\int \XinitIS{0}(\rmd x)\ewght{0}{2}(x)\mathbb{E}\left[A_N(x)\right]\eqsp,
\end{align*}
where
\[
A_N(x)\eqdef\left|\sum_{j=1}^N \frac{\ebackwght{t}{T}{j}}{\backsumwght{t}{T}} \tilde{h}_{0,t}(x,\ebackpart{t}{T}{j})-\backDist[][\gamma]{t}{T}[ \tilde{h}_{0,t}(x,\cdot)]\right|\eqsp.
\]
By Proposition~\ref{prop:exponential-inequality-backward}, there exist $B_{t|T}$ and $C_{t|T}$ such that for all $x\in\Xset$,
\begin{multline}
\label{eq:L1fromExp}
\mathbb{E}\left[A_N(x)\right] = \int_0^{\infty}\mathbb{P}\left(A_N(x)\ge \varepsilon\right)\rmd \varepsilon \\
\le B_{t|T} \int_0^{\infty}\rme^{-C_{t|T} N \epsilon^2 / \oscnorm{h}^2}\rmd\varepsilon\le D_{t|T}\oscnorm{h}N^{-1/2}\eqsp,
\end{multline}
which shows that
\[
\sum_{i=1}^N \CPE{U_{N,i}^2}{\mcg{N,i-1}}\plim_{N \to \infty}  \sigma^2_{0,t|T}[h]
\]
and concludes the proof of  \eqref{eq:condRD-1}. For all $N\ge 1$,
\begin{equation*}
\{ |U_{N,i}| \geq \varepsilon \}  \subseteq \left\{ \ewght{0}{i} \geq \varepsilon N^{1/2} \oscnorm{h}^{-1} \right\} \eqsp,
\end{equation*}
which implies that
\[
\sum_{i=1}^N \CPE{U_{N,i}^2 \1\{|U_{N,i}| \geq \varepsilon\}}{\mcg{N,i-1}} \\
\leq \oscnorm{h}^2\int \XinitIS{0}(\rmd x) \ewght{0}{2}(x) \1 \left\{ \ewght{0}{}(x) \geq N^{1/2}  \oscnorm{h}^{-1}\right\}
\]
and \eqref{eq:condRD-3} follows by letting $N \to \infty$.  Note that
\[
N^{-1}\sumwght{0}\plim_{N\to\infty}  \int \Xinit(\rmd x_0) g_0(x_0)\eqsp,
\]
which shows \eqref{eq:asymVarjoint} since
\begin{align*}
\asymVarJoint{0,t}{T}{\tilde{h}_{0,t}} &= \left(\int \Xinit(\rmd x_0) g_0(x_0)\right)^{-2}\\
&\hspace{2.6cm}\times\left(\backasymVar[\gamma]{t}{T}{H_{0,t}} + \int \XinitIS{0}(\rmd x)\ewght{0}{2}(x)\backDist[][\gamma]{t}{T}^2[\tilde{h}_{0,t}(x,\cdot)]\right)\eqsp,\\
&=   \backasymVar[\gamma]{t}{T}{\int \filt{\Xinit,0}(\rmd x_0)\tilde{h}_{0,t}(x_0,\cdot)} + \asymVar[\Xinit]{0}{\int \backDist[][\gamma]{t}{T}(\rmd x_t)\tilde{h}_{0,t}(\cdot,x_t)} \eqsp.
\end{align*}
Assume now that the result holds for some $s-1$.
Write $\tilde{h}_{s,t}\eqdef h - \filt{\Xinit,s}\otimes \backDist[][\gamma]{t}{T}[h]$ and set, for $i \in \{1,\dots,N\}$,
\begin{equation*}
U_{N,i} \eqdef N^{-1/2}\swght{s}{i}\sum_{j=1}^N \frac{\ebackwght{t}{T}{j}}{\backsumwght{t}{T}} \tilde{h}_{s,t}(\epartpred{s}{i},\ebackpart{t}{T}{j})\eqsp.
\end{equation*}
Then,
\[
\sqrt{N}\left(\sum_{i,j=1}^N \frac{\swght{s}{i}}{\sumwght{s}}\frac{\ebackwght{t}{T}{j}}{\backsumwght{t}{T}} h(\epartpred{s}{i},\ebackpart{t}{T}{j}) - \filt{\Xinit,0}\otimes \backDist[][\gamma]{t}{T}[h]
\right) = \left(\sumwght{s}/N\right)^{-1}\sum_{i=1}^N U_{N,i}\eqsp.
\]
Define, for $1 \leq i \leq N$,
\[
\mcg{N,i}\eqdef \sigma \left( \left\{ \epart{s}{j} \right\}_{j=1}^i, \left\{ \epart{u}{\ell} \right\}_{\ell=1}^N, \left\{ \ebackpart{v}{T}{j} \right\}_{j=1}^N,  1 \leq u < s, t \leq v \leq T\right) \eqsp.
\]
Then,
$$
\sum_{i=1}^N \CPE{U_{N,i}}{\mcg{N,i-1}} = \left( \filt{\chi,s-1}^N[\adjfuncforward{s}] \right)^{-1} N^{1/2}\sum_{i,j=1}^N\frac{\swght{s-1}{i}}{\sumwght{s-1}}\frac{\ebackwght{t}{T}{j}}{\backsumwght{t}{T}} H_s(\epartpred{s-1}{i}, \ebackpart{t}{T}{j})\eqsp,
$$
where
\begin{equation}
\label{eq:def:Hst}
H_{s,t}(x_{s-1},x_t) \eqdef \int q(x_{s-1},x)g_s(x)\tilde{h}_{s,t}(x,x_t)\rmd x\eqsp.
\end{equation}
Since $\filt{\chi,s-1}\otimes \backDist[][\gamma]{t}{T}[H_{s,t}]=0$,
by the induction assumption,
$$
\sum_{i=1}^N \CPE{U_{N,i}}{\mcg{N,i-1}} \dlim_{N \to \infty} \mathcal{N}\left(0, \asymVarJoint{s-1,t}{T}{H_{s,t}}/ \filt{\chi,s-1}^2[\adjfuncforward{s}] \right) \eqsp.
$$
We will now prove that
\[
\CPE{\exp\left(\rmi u \sum_{i=1}^N \{U_{N,i}-\CPE{U_{N,i}}{\mcg{N,i-1}}\} \right)}{\mcg{N,0}} \plim_{N \to \infty} \exp\left(-\frac{u^2 \sigma^2_{s,t|T}[h]}{2}\right) \eqsp,
\]
where
\begin{align*}
\sigma^2_{s,t|T}[h] &\eqdef \filt{\chi,s-1}[\adjfuncforward{s}]^{-1}\filt{\Xinit,s-1}\left[f_{s-1,t}\right]\eqsp,\\
f_{s-1,t}(x_{s-1}) &\eqdef \int q(x_{s-1},x_s)\ewghtfunc{s}(x_{s-1},x_s)
\backDist[][\gamma]{t}{T}^2[\tilde{h}_{s,t}(x_s,\cdot)] g_ s(x_s)\rmd x_s\eqsp.
\end{align*}
This is done using again \cite[Theorem~A.3]{douc:moulines:2008} and proving that \eqref{eq:condRD-1} and \eqref{eq:condRD-3} hold with $\sigma^2_{0,t|T}[h]$ replaced by $\sigma^2_{s,t|T}[h]$. Note that
\begin{equation*}
 \sum_{i=1}^N \CPE{U_{N,i}}{\mcg{N,i-1}}^2= \left(\sum_{i, j=1}^N\frac{\swght{s-1}{i}}{\sumwght{s-1}}\frac{\ebackwght{t}{T}{j}}{\backsumwght{t}{T}} H_{s,t}(\epartpred{s-1}{i}, \ebackpart{t}{T}{j})\right)^2/(\filt{\chi,s-1}^N[\adjfuncforward{s}])^2 \eqsp,
\end{equation*}
which converges in probability to $0$ by Theorem~\ref{thm:Hoeffding-Two-Populations} and the fact that $\filt{\chi,s-1}\otimes \backDist[][\gamma]{t}{T}[H_{s,t}]=0$. In addition,
\begin{align*}
\sum_{i=1}^N \CPE{U_{N,i}^2}{\mcg{N,i-1}}&= \left(\filt{\chi,s-1}^N[\adjfuncforward{s}]\right)^{-1}\sum_{i=1}^N\frac{\ewght{s-1}{i}}{\sumwght{s-1}}\int  \frac{q^2(\epart{s-1}{i}, x_{s}) g_s(x_s)}{\adjfuncforward{s}(\epart{s-1}{i})\kissforward{s}{s}(\epart{s-1}{i},x_s)} \\
&\hspace{4cm}\times \left(\backDist[][\gamma]{t}{T}^N[\tilde{h}_{s,t}(x_s,\cdot)]\right)^2 g_s(x_s) \rmd x_s\eqsp,\\
 &= \left(\filt{\chi,s-1}^N[\adjfuncforward{s}]\right)^{-1}\sum_{i=1}^N\frac{\ewght{s-1}{i}}{\sumwght{s-1}}\int  \ewghtfunc{s}(\epart{s-1}{i},x_s)q(\epart{s-1}{i}, x_{s}) \\
&\hspace{4cm}\times \left(\backDist[][\gamma]{t}{T}^N[\tilde{h}_{s,t}(x_s,\cdot)]\right)^2 g_s(x_s) \rmd x_s\eqsp,\\
&=\left(\filt{\chi,s-1}^N[\adjfuncforward{s}]\right)^{-1}\filt{\Xinit,s-1}^N\left[ f_{s-1,t}^N\right]\eqsp,
\end{align*}
where
\[
f_{s-1,t}^N(x_{s-1}) \eqdef \int q(x_{s-1},x_s)\ewghtfunc{s}(x_{s-1},x_s)
\left(\backDist[][\gamma]{t}{T}^N[\tilde{h}_{s,t}(x_s,\cdot)]\right)^2 g_ s(x_s)\rmd x_s\eqsp.
\]
First note that $\filt{\chi,s-1}^N[\adjfuncforward{s}]\plim_{N \to \infty}\filt{\chi,s-1}[\adjfuncforward{s}]$ and write
\[
\left|\filt{\Xinit,s-1}^N\left[ f_{s-1,t}^N\right] - \filt{\Xinit,s-1} \left[ f_{s-1,t}\right]\right| \le A^N_{s,t} + B^N_{s,t} \eqsp,
\]
where $A^N_{s,t} \eqdef|\filt{\Xinit,s-1}^N[ f_{s-1,t}^N]  - \filt{\Xinit,s-1}^N[f_{s-1,t}]|$ 
 and $B^N_{s,t}  \eqdef |\filt{\Xinit,s-1}^N [f_{s-1,t}]-\filt{\Xinit,s-1}[f_{s-1,t}]|$. As $(\ewght{s-1}{i},\epart{s-1}{i})_{i=1}^N$ and $(\ebackwght{t}{T}{j},\ebackpart{t}{T}{j})_{j=1}^N$ are independent,
\[
\mathbb{E}\left[A^N_{s,t}\right] \le \esssup{\ewghtfunc{s}}\esssup{g_s}\mathbb{E}\left[\sum_{i=1}^N\frac{\ewght{s-1}{i}}{\sumwght{s-1}}\int  q(\epart{s-1}{i}, x_{s})\mathbb{E}\left[\left|\Delta  \backDist[][\gamma]{t}{T}^N[\tilde{h}_{s,t}](x_s)\right|\right]\rmd x_s\right]\eqsp,
\]
where $\Delta  \backDist[][\gamma]{t}{T}^N[\tilde{h}_{s,t}](x_s) \eqdef (\backDist[][\gamma]{t}{T}^N[\tilde{h}_{s,t}(x_s,\cdot)])^2- \backDist[][\gamma]{t}{T}^2[\tilde{h}_{s,t}(x_s,\cdot)]$.
Following the same steps as in \eqref{eq:L1fromExp}, there exists $D_{t|T}$ such that
\[
\mathbb{E}\left[\left|\Delta  \backDist[][\gamma]{t}{T}^N[\tilde{h}_{s,t}](x_s)\right|\right] \le 2D_{t|T}\oscnorm{h}^2/\sqrt{N}\eqsp,
\]
which yields
\begin{align*}
\mathbb{E}\left[A^N_{s,t}\right] &\le 2\oscnorm{h}^2\esssup{\ewghtfunc{s}}\esssup{g_s}D_{T|t}\mathbb{E}\left[\sum_{i=1}^N\frac{\ewght{s-1}{i}}{\sumwght{s-1}}\int  q(\epart{s-1}{i}, x_{s})\rmd x_s\right] /\sqrt{N}\eqsp,\\
&\le 2\oscnorm{h}^2\esssup{\ewghtfunc{s}}\esssup{g_s}D_{T|t}/\sqrt{N}
\end{align*}
and $\mathbb{E}\left[A^N_{s,t}\right]\longrightarrow_{N\to \infty}0 $. On the other hand, as $\oscnorm{f_{s-1,t}}\le \oscnorm{h}^2\esssup{\ewghtfunc{s}}\esssup{g_s}$, $B^N_{s,t}\plim_{N \to \infty}0$ by Proposition~\ref{prop:exponential-inequality-forward}.
Finally, the tightness condition \eqref{eq:condRD-3} holds since $|U_{N,i}| \leq N^{-1/2}\esssup{\ewghtfunc{s}}\oscnorm{h}$. Note that,
\[
N^{-1}\sumwght{s}\plim_{N\to\infty} \filt{\chi,s-1}\left[\int q(\cdot,x_s)g_s(x_s)\rmd x_s\right]/\filt{\chi,s-1}[\adjfuncforward{s}]\eqsp.
\]
Therefore \eqref{eq:asymVarjoint} holds with
\begin{align*}
\asymVarJoint{s,t}{T}{\tilde{h}_{s,t}} &= \frac{\filt{\chi,s-1}^2[\adjfuncforward{s}]}{\filt{\chi,s-1}^2\left[\int q(\cdot,x_s)g_s(x_s)\rmd x_s\right]}\left\{\frac{\asymVarJoint{s-1,t}{T}{H_{s,t}}}{\filt{\chi,s-1}^2[\adjfuncforward{s}]}  +  \frac{\filt{\Xinit,s-1}\left[f_{s-1,t}\right]}{\filt{\chi,s-1}[\adjfuncforward{s}]}\right\}\eqsp,\\
&=   \frac{\asymVarJoint{s-1,t}{T}{H_{s,t}}}{\filt{\chi,s-1}^2\left[\int q(\cdot,x_s)g_s(x_s)\rmd x_s\right]}  +  \frac{\filt{\Xinit,s-1}\left[f_{s-1,t}\right]\filt{\chi,s-1}[\adjfuncforward{s}]}{\filt{\chi,s-1}^2\left[\int q(\cdot,x_s)g_s(x_s)\rmd x_s\right]}\eqsp,
\end{align*}
where, by induction assumption,
\begin{multline*}
\asymVarJoint{s-1,t}{T}{H_{s,t}} = \asymVar[\Xinit]{s-1}{\int \backDist[][\gamma]{t}{T}(\rmd x_t)H_{s,t}(\cdot,x_t)} \\
+ \backasymVar[\gamma]{t}{T}{\int \filt{\Xinit,s-1}(\rmd x_{s-1})H_{s,t}(x_{s-1},\cdot)}\eqsp.
\end{multline*}
The proof is completed upon noting that
\begin{align*}
\frac{\int \filt{\Xinit,s-1}(\rmd x_{s-1})H_s(x_{s-1},\cdot)}{\filt{\chi,s-1}\left[\int q(\cdot,x_s)g_s(x_s)\rmd x_s\right]} &= \frac{\int \filt{\Xinit,s-1}(\rmd x_{s-1})q(x_{s-1},x_s)g_s(x_s)\tilde{h}_{s,t}(x_s,\cdot)\rmd x_s}{\filt{\chi,s-1}\left[\int q(\cdot,x_s)g_s(x_s)\rmd x_s\right]}\eqsp,\\
&= \int \filt{\Xinit,s}(\rmd x_{s})\tilde{h}_{s,t}(x_s,\cdot)
\end{align*}
and, by Proposition~\ref{prop:CLT-forward},
\begin{multline*}
\asymVar[\Xinit]{s-1}{\frac{\int \backDist[][\gamma]{t}{T}(\rmd x_t)H_{s,t}(\cdot,x_t)}{\filt{\chi,s-1}\left[\int q(\cdot,x_s)g_s(x_s)\rmd x_s\right]}} + \frac{\filt{\Xinit,s-1}\left[f_{s-1,t}\right]\filt{\chi,s-1}[\adjfuncforward{s}]}{\filt{\chi,s-1}^2\left[\int q(\cdot,x_s)g_s(x_s)\rmd x_s\right]} \\
= \asymVar[\Xinit]{s}{\int \backDist[][\gamma]{t}{T}(\rmd x_t)\tilde{h}_{s,t}(\cdot,x_t)}\eqsp.
\end{multline*}
\subsection{Proof of Theorem~\ref{thm:CLT-Fearnhead-smoother}}
\label{proof:thm:CLT-Fearnhead-smoother}
Write $\tilde{h}_{s,T} =  h - \post{\Xinit,s}{T}(h)$. Note that
\[
\sqrt{N}\sum_{i=1}^N \frac{\smwght{s}{T}{i}}{\tilde{\Omega}_{s|T}} \tilde{h}_{s,T}(\smpart{s}{T}{i}) = \left(\tilde{\Omega}_{s|T}/N\right)^{-1}\sum_{i=1}^N U_{N,i}\eqsp,
\]
where $U_{N,\ell} \eqdef N^{-1/2} \smwght{s}{T}{\ell} \tilde{h}_{s,T}(\smpart{s}{T}{\ell})$. Set, for $i \in \{1, \dots, N\}$,
\begin{multline*}
\mcg{N,i} \eqdef \sigma\left\{ \{ (\smwght{s}{T}{\ell},\smpart{s}{T}{\ell}) \}_{\ell=1}^i \eqsp,\eqsp\{ (\ewght{u}{\ell}, \epart{u}{\ell}) \}_{\ell=1}^N, u =0, \dots, s-1,\right.\\
\left. \{ (\ebackwght{v}{T}{\ell},\ebackpart{v}{T}{\ell}) \}_{\ell=1}^N, v =s+1, \dots, T\right\} \eqsp.
\end{multline*}
By the proof of Theorem \ref{thm:Hoeffding-Fearnhead},
\[
N^{-1} \tilde{\Omega}_{s|T} \plim_{N \to \infty}
\frac{\filt{\Xinit,s-1} \otimes \backDist{s+1}{T} \left[ \int q^{[2]}(\cdot,x) g_s(x) \rmd x \odot \fakeprior_{s+1}^{-1}\right]}
{\filt{\Xinit,s-1} \otimes \backDist{s+1}{T} \left[ \adjfunc[smooth]{s}{T}{} \odot \fakeprior_{s+1}^{-1} \right]} \eqsp.
\]
The proof therefore amounts to establish a CLT for $\sum_{\ell=1}^N U_{N,\ell}$ and then to use Slutsky Lemma.
The limit distribution of $\sum_{\ell=1}^N U_{N,\ell}$ is again obtained using  the invariance principle for triangular array of dependent random variables derived in \cite{douc:moulines:2008}. As
\begin{multline*}
\sum_{i=1}^N \CPE{U_{N,i}}{\mcg{N,i-1}}\\
 = \sqrt{N} \frac{\sum_{i,j=1}^N \ewght{s-1}{i} \ebackwght{s+1}{T}{j} \int q^{[2]}(\cdot; x_s) g_s(x_s) \tilde{h}_{s,T}(x_s)\rmd x_s\odot \fakeprior_{s+1}^{-1}(\epart{s-1}{i},\ebackpart{s+1}{T}{j}) }
{\sum_{i,j=1}^N \ewght{s-1}{i} \ebackwght{s+1}{T}{j} \fakeprior_{s+1}^{-1}(\ebackpart{s+1}{T}{j}) \adjfunc[smooth]{s}{T}{\epart{s-1}{i},\ebackpart{s+1}{T}{j}}}\eqsp,
\end{multline*}
it follows from Theorems \ref{thm:Hoeffding-Two-Populations} and \ref{thm:CLT-Two-Populations} that
\[
\sum_{i=1}^N \CPE{U_{N,i}}{\mcg{N,i-1}} \plim_{N \to \infty} \mathcal{N}\left(0, \frac{\Sigma_s[\tilde{h}_{s,T}]}{\left(\filt{\Xinit,s-1} \otimes \backDist{s+1}{T} [\adjfunc[smooth]{s}{T}{}\odot\gamma_{s+1}^{-1}] \right)^2} \right) \eqsp.
\]
Using that
\begin{equation*}
\post{\Xinit,s}{T}[\tilde{h}_{s,T}]= \frac{\filt{\Xinit,s-1} \otimes \backDist{s+1}{T}\left[ \int q^{[2]}(\cdot;x) g_s(x) \tilde{h}_{s,T}(x) \rmd x\right] \odot \gamma_{s+1}^{-1}}
{\filt{\Xinit,s-1} \otimes \backDist{s+1}{T}\left[ \int q^{[2]}(\cdot;x) g_s(x)\rmd x\right] \odot \gamma_{s+1}^{-1} }= 0 \eqsp,
\end{equation*}
Theorem \ref{thm:Hoeffding-Two-Populations} yields
\begin{align*}
&\sum_{i=1}^N\CPE{U_{N,i}}{\mcg{N,i-1}}^2 = \\
&\hspace{1cm}\left(\frac{\sum_{i,j=1}^N \ewght{s-1}{i} \ebackwght{s+1}{T}{j} \fakeprior_{s+1}^{-1}(\ebackpart{s+1}{T}{j}) \int q^{[2]}(\epart{s-1}{i},\ebackpart{s+1}{T}{j}; x) g_s(x) \tilde{h}_{s,T}(x)}
{\sum_{i,j=1}^N \ewght{s-1}{i} \ebackwght{s+1}{T}{j} \fakeprior_{s+1}^{-1}(\ebackpart{s+1}{T}{j}) \adjfunc[smooth]{s}{T}{\epart{s-1}{i},\ebackpart{s+1}{T}{j}}} \right)^2\\
&\hspace{1cm}\plim_{N \to \infty} \left( \frac{\filt{\Xinit,s-1} \otimes \backDist{s+1}{T}\left[ \int q^{[2]}(\cdot;x) g_s(x) \tilde{h}_{s,T}(x) \rmd x\odot \fakeprior_{s+1}^{-1}\right]}
{\filt{\Xinit,s-1} \otimes \backDist{s+1}{T} [\adjfunc[smooth]{s}{T}{} \odot \fakeprior_{s+1}^{-1}]} \right)^2 = 0 \eqsp.
\end{align*}
Similarly, using again Theorem \ref{thm:Hoeffding-Two-Populations},
\begin{multline*}
\sum_{i=1}^N \CPE{U_{N,i}^2}{\mcg{N,i-1}} \plim_{N \to \infty} \\
\frac{\filt{\Xinit,s-1} \otimes \backDist{s+1}{T}\left[ \int \smwghtfunc{s}{T}(\cdot;x) q^{[2]}(\cdot;x) g_s(x) \tilde{h}_{s,T}^2(x) \rmd x \odot \fakeprior_{s+1}^{-1}\right]}
{\filt{\Xinit,s-1} \otimes \backDist{s+1}{T} [\adjfunc[smooth]{s}{T}{} \odot \fakeprior_{s+1}^{-1}]} \eqsp.
\end{multline*}
Since under A\ref{assum:borne-FFBS}, $|U_{N,i}| \leq N^{-1/2} \esssup{\smwghtfunc{s}{T}} \oscnorm{h}$, for any $\epsilon > 0$, $$\sum_{i=1}^N \CPE{U^2_{N,i} \1 \{ |U_{N,i}| \geq \epsilon \}}{\mcg{N,i-1}} \plim_{N \to \infty} 0 \eqsp,$$
which concludes the proof.

\clearpage

\appendix

\section{Exponential deviation inequalities for the forward filter and the backward information filter}
\label{sec:appendix:expdev}The following result is proved in \cite{douc:garivier:moulines:olsson:2011}.
\begin{lem}
\label{lem:hoeffding:ratio}
Assume that $a_N$, $b_N$, and $b$ are random variables defined on the same probability space such that there exist positive constants $\beta$, $B$, $C$, and $M$ satisfying
\begin{enumerate}[(i)]
    \item \label{it:bound:ab} $|a_N/b_N|\leq M$, $\mathbb{P}$-a.s.\ and  $b \geq \beta$, $\mathbb{P}$-a.s.,
    \item \label{it:hoeff:a} For all $\epsilon>0$ and all $N\geq1$, $\mathbb{P} \left[ |a_N|>\epsilon \right]\leq B \rme^{-C N \left(\epsilon/M\right)^2}$,
    \item \label{it:hoeff:b} For all $\epsilon>0$ and all $N\geq1$, $\mathbb{P}\left[|b_N-b|>\epsilon \right]\leq B \rme^{-C N \epsilon^2}$.
\end{enumerate}
Then, for all $\varepsilon>0$,
\[
\mathbb{P}\left[ \left| \frac{a_N}{b_N} \right| > \epsilon \right] \leq B \exp{\left\{-C N \left(\frac{\epsilon \beta}{2M} \right)^2 \right\}} \eqsp.
\]
\end{lem}

Proposition~\ref{prop:exponential-inequality-forward} provides an exponential deviation inequality for the forward filter and is proved in \cite{douc:garivier:moulines:olsson:2011}.
\begin{prop}
\label{prop:exponential-inequality-forward}
Assume that A\ref{assum:bound-likelihood} and A\ref{assum:borne-FFBS} hold for some $T > 0$. Then, for all $s\geq 1$,
there exist $0 <B_s, C_s <\infty$ such that for all $N\ge 1$, $\epsilon > 0$, and all $h \in \functionset[b]{X}$,
\begin{equation*}
\PP\left( \left| \sumwght{s}^{-1 } \sum_{i=1}^N \swght{s}{i} h(\epartpred{s}{i}) - \filt{\Xinit,s}[h] \right| \geq \epsilon \right) \leq B_s \rme^{-C_s N \epsilon^2 / \oscnorm{h}^2}\eqsp.
\end{equation*}
\end{prop}

Proposition~\ref{prop:exponential-inequality-backward} provides an exponential inequality for the backward information filter $\backDist[][\gamma]{t}{T}$ and its unnormalized approximation. Its proof is similar to the proof \cite[Theorem~5]{douc:garivier:moulines:olsson:2011} and is omitted.

\begin{prop}
\label{prop:exponential-inequality-backward}
Assume that A\ref{assum:bound-likelihood} and A\ref{assum:borne-TwoFilters} hold for some $T > 0$. Then, for all $0\le t\leq T$,
there exist $0<B_{t|T}, C_{t|T}<\infty$  such that for all $N\ge 1$, $\epsilon > 0$, and all $h \in \functionset[b]{X}$,
\begin{equation*}
 \label{eq:Hoeffding-Backward-Filtering-normalized}
\PP\left[ \left| \backsumwght{t}{T}^{-1} \sum_{i=1}^N \ebackwght{t}{T}{i} h(\ebackpart{t}{T}{i}) - \backDist[][\gamma]{t}{T}[h] \right| \geq \epsilon \right] \leq B_{t|T} \rme^{-C_{t|T} N \epsilon^2 / \oscnorm{h}^2}  \eqsp.
\end{equation*}
\end{prop}

\section{Asymptotic normality of the forward filter and the backward information filter}
\label{sec:appendix:TCL}

Proposition~\ref{prop:CLT-forward} provides a CLT for the weighted particles $\{  (\ewght{s}{i},\epart{s}{i})\}_{i=1}^N$ approximating the filtering distribution $\filt{\Xinit,s}$ and is proved for instance in \cite{delmoral:2004}.
\begin{prop}
\label{prop:CLT-forward}
Assume that A\ref{assum:bound-likelihood} and A\ref{assum:borne-FFBS} hold for some $T > 0$. Then, for all $0\le s\le T$ and all $h \in \functionset[b]{X}$,
\[
N^{1/2} \left(\sum_{i=1}^N \frac{\ewght{s}{i}}{\sumwght{s}} h(\epart{s}{i}) - \filt{\Xinit,s}[h]\right) \dlim_{N \to \infty} \mathcal{N}\left(0,\asymVar[\Xinit]{s}{h - \filt{\Xinit,s}[h]}\right)\eqsp,
\]
where 
\begin{align*}
\asymVar[\Xinit]{0}{h} &\eqdef \frac{ \int \XinitIS{0}(\rmd x_0)\ewght{0}{2}(x_0)h^2(x_0)}{\left(\int \XinitIS{0}(\rmd x_0)\ewght{0}{}(x_0)\right)^{2}}\quad\mbox{and for all\;} s\ge 1\eqsp,\\
\asymVar[\Xinit]{s}{h} &\eqdef \frac{\asymVar[\Xinit]{s-1}{\int q(\cdot,x_{s})g_s(x_s)h(x_s)\rmd x_s}}{\filt{\chi,s-1}^2\left[\int q(\cdot,x_s)g_s(x_s)\rmd x_s\right]}\\
&\hspace{2cm}+\frac{\filt{\Xinit,s-1}\left[\int \ewght{s}{}(\cdot,x_s)q(\cdot,x_s)g_s(x_s)h^2(x_s)\rmd x_s\right]\filt{\chi,s-1}[\adjfuncforward{s}]}{\filt{\chi,s-1}^2\left[\int q(\cdot,x_s)g_s(x_s)\rmd x_s\right]}\eqsp.
\end{align*}
\end{prop}
Proposition~\ref{prop:CLT-backward} provides a CLT for the weighted particles $\{  (\ebackwght{t}{T}{j},\ebackpart{t}{T}{j})\}_{j=1}^N$ approximating the backward information filter. Its proof follows the same lines as the proof of Proposition~\ref{prop:CLT-forward} and is omitted for brevity.
\begin{prop}
\label{prop:CLT-backward}
Assume that A\ref{assum:bound-likelihood} and A\ref{assum:borne-TwoFilters} hold. Then, for all $0\le t\le T$ and all $h \in \functionset[b]{X}$,
\[
N^{1/2} \left(\sum_{j=1}^N \frac{\ebackwght{t}{T}{j}}{\backsumwght{t}{T}} h(\ebackpart{t}{T}{j})- \backDist[][\gamma]{t}{T}[h]\right) \dlim_{N \to \infty} \mathcal{N}\left(0,\backasymVar[\gamma]{t}{T}{h-\backDist[][\gamma]{t}{T}[h]}\right)\eqsp,
\]
where 
\begin{align*}
\backasymVar[\gamma]{T}{T}{h} &\eqdef \frac{ \int \ebackinit_T(\rmd x_T)\ebackwght{T}{T}{2}(x_T)h^2(x_T)}{\left( \int \ebackinit_T(\rmd x_T)\ebackwght{T}{T}{}(x_T)\right)^{2}}\quad\mbox{and for all\;} t\le T-1\eqsp,\\
\backasymVar[\gamma]{t}{T}{h} &\eqdef \frac{\backasymVar[\gamma]{t+1}{T}{\int\gamma_t(x_t)g_t(x_t)q(x_t,\cdot)\gamma^{-1}_{t+1}(\cdot)h(x_t)\rmd x_t}}{\backDist[][\gamma]{t+1}{T}^2\left[\int\gamma_t(x_t)g_t(x_t)q(x_t,\cdot)\gamma^{-1}_{t+1}(\cdot)\rmd x_t\right]}\\
&\hspace{.5cm}+\frac{\backDist[][\gamma]{t+1}{T}\left[\int\backweight_t(x_t,\cdot)q(x_t,\cdot)g_t(x_t)\gamma_t(x_t)\gamma^{-1}_{t+1}(\cdot)h^2(x_t)\rmd x_t\right]\backDist[][\gamma]{t+1}{T}\left[\adjfunc{t}{T}{}\gamma_{t+1}^{-1}\right]}{\backDist[][\gamma]{t+1}{T}^2\left[\int\gamma_t(x_t)g_t(x_t)q(x_t,\cdot)\gamma^{-1}_{t+1}(\cdot)\rmd x_t\right]}\eqsp.
\end{align*}
\end{prop}

\bibliographystyle{plain}
\bibliography{TwoFilterbib}

\end{document}